\documentclass[11pt]{amsart}

\usepackage{parskip}
\usepackage{todonotes}
\usepackage{caption}
\usepackage{subcaption}
\usepackage{graphicx}
\graphicspath{ {./} }

\usepackage{appendix}
\usepackage{comment}

\usepackage{amsmath}
\usepackage{amssymb}
\usepackage{amsthm}
\usepackage[utf8]{inputenc}
\usepackage{scrextend}
\usepackage{bbm}
\usepackage{stmaryrd}
\usepackage{geometry}
\usepackage{enumerate}
\usepackage{tikz, pgffor}
\usepackage{hyperref}
\usepackage{mathtools}
\usepackage{amsfonts}
\usepackage[
backend=bibtex,
style=alphabetic
]{biblatex}
\newcommand{\eps}{\varepsilon}
\renewcommand{\{}{\left\lbrace}
  \renewcommand{\}}{\right\rbrace}
\let\div\relax
\DeclareMathOperator{\div}{div}

\DeclareMathOperator*{\dist}{dist}
\DeclareMathOperator*{\tr}{tr}
\DeclareMathOperator{\supp}{supp}
\DeclareMathOperator*{\Lip}{Lip}
\DeclareMathOperator*{\esssup}{ess\,sup}
\DeclareMathOperator*{\argmin}{arg\,min}
\DeclareMathOperator{\Dom}{Dom}

\numberwithin{equation}{section}

\theoremstyle{plain}
\newtheorem{theorem}{Theorem}[section]

\newtheorem{lemma}[theorem]{Lemma}
\newtheorem{corollary}[theorem]{Corollary}
\newtheorem{proposition}[theorem]{Proposition}

\theoremstyle{definition}
\newtheorem{definition}[theorem]{Definition}
\newtheorem{step}{Step}

\theoremstyle{remark}
\newtheorem{remark}[theorem]{Remark}

\title[Hele--Shaw flow as limit of Cahn--Hilliard equation]{The Hele--Shaw flow as the sharp interface limit of the Cahn--Hilliard
  equation with disparate mobilities}

\author{Milan Kroemer}

\author{Tim Laux}
\address{Hausdorff Center for Mathematics and Institute for Applied Mathematics, University of Bonn, Villa Maria, Endenicher Allee 62, 53115 Bonn, Germany}

\email{ \nolinkurl{{ milan.kroemer, tim.laux}@hcm.uni-bonn.de} }
\begin{document}

\maketitle
\begin{abstract}
	In this paper, we study
	the sharp interface limit for solutions of the
	Cahn--Hilliard equation with disparate mobilities.
	This means that the mobility function degenerates in one of the two energetically favorable configurations, suppressing the diffusion in that phase. 
	First, we construct suitable weak solutions to this Cahn--Hilliard equation.
	Second, we prove precompactness of these solutions under natural assumptions on the initial data.
	Third, under an additional energy convergence assumption, we show that the sharp interface limit is a distributional solution to the Hele--Shaw flow with optimal energy-dissipation rate.
	
		\medskip
	
	
		\medskip
	
	\noindent \textbf{Mathematical Subject Classification (MSC 2020)}: 
  53E40, 35K65 (primary), 35K55, 82C26, 76D27 (secondary)
\end{abstract}



\section{Introduction}\label{sec:introduction}

	The Hele--Shaw cell is made of two parallel horizontal sheets
	which are separated by a thin gap of width $b$.
	Between the two sheets, a viscous fluid fills an almost cylindrical domain.
	As the spacing $b$ between the plates vanishes, one considers the lower-dimensional cross-section $\Omega$ of the fluid.
	Formal arguments suggest that this limit is governed by the Hele--Shaw flow, see~\eqref{eq:velocity-eq}--\eqref{eq:conservation-eq} below for the precise formulation.
	Otto~\cite{Otto1998} studied this reduced model for a ferrofluid in the presence of an external magnetic field to explain patterns observed in experiments~\cite{Jackson}.
	There is yet another, less classical way in which the Hele--Shaw flow arises in a singular limit, namely as the sharp interface limit of a Cahn--Hilliard equation.
	This is suggested by formal matched asymptotic expansions~\cite{Glasner2003}. 
	The goal of this paper is to rigorously justify the connection between these two models. 
	
	The Cahn--Hilliard equation is a fundamental phase-field model describing the phase separation for preserved order parameters.
	We are interested in the case of disparate mobilities, i.e., the case when the mobility function vanishes in one of the two stable states but is non-degenerate in the other one.
	In that case, this degenerate parabolic PDE has a rich structure: it is the gradient flow of the Cahn--Hilliard energy in the Wasserstein space of probability measures.
	Based on this gradient-flow structure, we first construct weak solutions to these degenerate Cahn--Hilliard equations and then analyze their convergence in the sharp-interface limit.

	Elliott and Garcke~\cite{Elliott1996} established the existence of solutions to such degenerate Cahn--Hilliard equations in a general setting; we also refer to \cite{Lisini2012} for a similar result.
	We propose here an alternative construction based on the Wasserstein gradient flow structure, which is mostly soft. 
	This is inspired by the seminal work of Jordan, Kinderlehrer, and Otto~\cite{Jordan1998}.
	Here, our energy is of higher order (as it depends on the gradient), and is therefore not geodesically convex as in the case of~\cite{Jordan1998}.
	For the most part of our proof, we do not rely on higher regularity as in~\cite{Elliott1996, Lisini2012}, and we are confident that some of these ideas will be useful in other situations as well.
	One interesting result in its own right is the computation of the first variation of the Dirichlet energy in Wasserstein space relying only on the natural $H^1$ regularity, see Lemma~\ref{lem:variation-dirichlet-energy}.
	This is used to show that any limit of minimizing movements (or the JKO scheme) is a weak solution of our Cahn--Hilliard equation.
	In addition, we show that this weak solution saturates the optimal energy-dissipation rate.
	Another crucial observation is that the first variation of the Cahn--Hilliard energy in Wasserstein space is in divergence form.
	This is well-known for domain variations given by the transport equation $\partial_s u_s + \xi \cdot \nabla u_s =0$ due to a result by Luckhaus and Modica~\cite{Luckhaus1989}.
	In contrast, variations in Wasserstein space are given by conservation laws of the form $\partial_s u_s + \nabla\cdot(u_s\xi)=0$.
	Our observation---which we already employ in the construction of our weak solution of the Cahn--Hilliard equation---is that also in this case, the first variation of the energy is in divergence form.
	This results in a stable notion of weak solutions and ultimately allows us to pass to the limit in our weak formulation and show that the limit is a weak solution of the Hele--Shaw flow under a typical assumption on the convergence of energies.	
	Similar sharp-interface limits have been studied in different settings, for example in the case of constant mobility~\cite{Alikakos1994, Chen1996}. 
	Our approach is different and inspired by the work of Simon and one of the authors~\cite{Laux2018} who derive the sharp-interface limit of a system of Allen--Cahn equations, a second-order version of our problem here. 
	On a conceptual level, our proof of the second main result is also similar to the work by Chambolle and one of the authors~\cite{Chambolle2021} who showed that the implicit time discretization of the Hele--Shaw flow produces varifold solutions which are slightly weaker than the solutions considered here. 
	Jacobs, Meszaros, and Kim~\cite{Jacobs2021} introduced a thresholding-type scheme, similar to this implicit time discretization and proved its convergence to a weak solution under an energy convergence assumption.

	We now state the setting in more detail and give an overview of our results.
	For a given (length) scale $\eps>0$ and a field $u\colon \mathbb{R}^d \to \mathbb{R}$ we consider the Cahn--Hilliard free energy
	\begin{equation}
	  \label{eq:def-energy-eps}
	  E_{\eps}(u)=\int_{\mathbb{R}^{d}}\frac{\eps}{2}|\nabla u|^{2}+\frac{1}{\eps}W(u)\,dx.
	\end{equation}
	For initial data $u_{\eps,0}$ we say that $u_\eps$ solves the Cahn--Hilliard equation if
	\begin{equation}
	  \label{eq:cahn-hilliard-eq-fix-eps-classic}
	    \partial_{t}u_{\eps}+\nabla\cdot
	    \left( m(u_{\eps})\left( -\nabla\frac{\delta E_\eps}{\delta u_\eps}
	      \right) \right)=0
	\end{equation}
	together with the initial condition $ u_\eps(\cdot,0)=u_{\eps,0}.$
	Here $m:\mathbb{R}\rightarrow[0,\infty)$ is the mobility function
	and $W:\mathbb{R}\rightarrow[0,\infty)$ is the standard double-well potential
	$W(s)=\frac{1}{4}s^{2}(s-1)^{2}.$ 
	In this work, we want to consider a density-dependent mobility function which is degenerate in one of the two phases, say, $u=0$.
	We focus on the prototypical example 
	$m(u)=u_+ = \max\{u,0\}.$
  \begin{figure}
    \centering
    \includegraphics[scale=0.4]{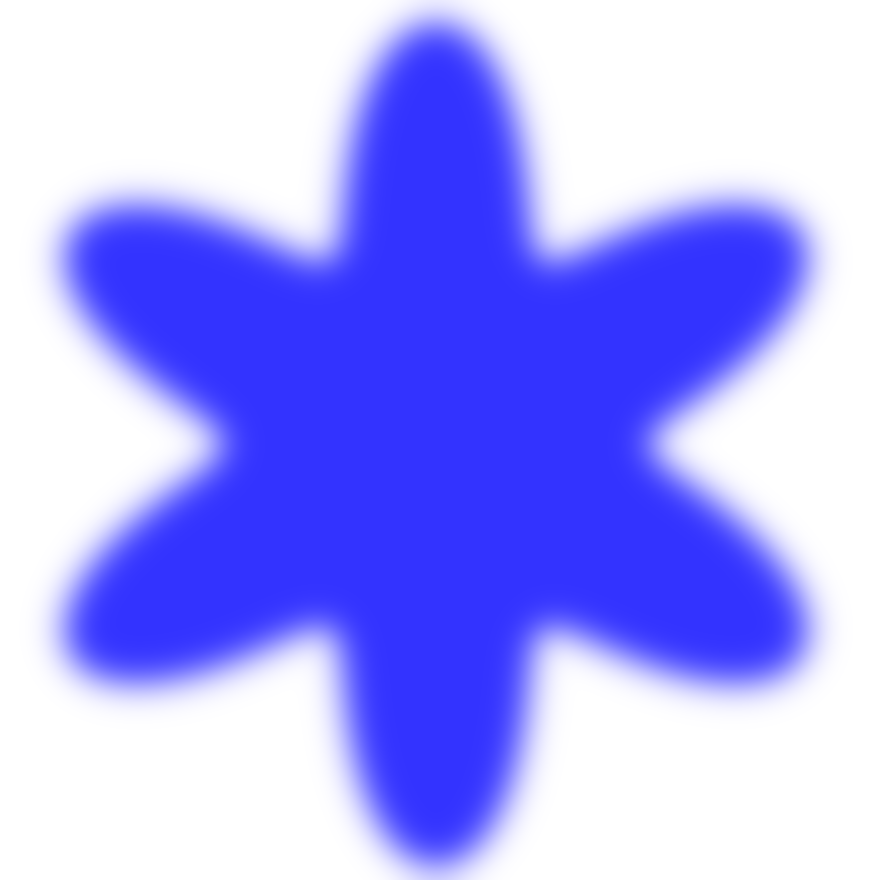}
    \caption{Diffuse interface}
    \label{fig:Diffuse}
  \end{figure}
	Heuristically, it is clear that diffuse interfaces of thickness $\sim \eps$ will develop in this model, see Figure~\ref{fig:Diffuse}.
	The main goal of this work is to understand the behavior of the solutions $u_\eps=u_\eps(x,t)$ in the singular limit
	$\eps\downarrow 0$.
	It is well-known since the work of Modica and Mortola~\cite{Modica1977} and Modica~\cite{Modica1987} that the Cahn--Hilliard energy $\Gamma$-converges to a multiple of the perimeter functional. 
	Therefore, our goal here can be formulated as extending this convergence to the corresponding gradient flows.

\begin{figure}[t!]
  \centering
  \begin{subfigure}[t]{.5\linewidth}
    \centering
    \begin{tikzpicture}[x={(-0.5cm,-0.5cm)}, y={(0.9659cm,-0.25882cm)},
      scale=1]
      \draw[fill, white!80!blue, opacity=0.8, domain=0:2*pi,samples=200] plot
      ({deg( \x) }:{2-(cos(deg(3*\x)))^2}); \draw[thick, black,
      domain=0:2*pi,samples=200] plot ({deg( \x) }:{2-(cos(deg(3*\x)))^2});
      \draw[dashed, black] (-2.5,-2.5) -- (-2.5,2.5)-- (2.5,2.5) --(2.5,-2.5) --
      (-2.5,-2.5);
      \draw[fill, gray, opacity=0.05] (-2.5,-2.5) -- (-2.5,2.5)--
      (2.5,2.5) --(2.5,-2.5) -- (-2.5,-2.5);
      \draw[thick] (-2.5,2.5) -- +(0,.15cm) node[pos=.5,anchor=west] {$b$};
      
      \foreach \n in {2.632, 6.7, 2.04, 3.28, 6.05, 7.26} \draw [thick, black]
      ({deg(\n*pi/4)}:{2-(cos(deg(3*\n*pi/4)))^2}) -- +(0,.15cm);
      
      \begin{scope}[xshift=0cm, yshift=.15cm]
        \draw[fill, white!80!blue, opacity=0.8, domain=0:2*pi,samples=200]
        plot ({deg( \x) }:{2-(cos(deg(3*\x)))^2});
        \draw[thick, black,domain=0:2*pi,samples=200]
        plot ({deg( \x) }:{2-(cos(deg(3*\x)))^2});
        \draw[dashed, black]
        (-2.5,-2.5) -- (-2.5,2.5)-- (2.5,2.5) --(2.5,-2.5) -- (-2.5,-2.5);
        \draw[fill, gray, opacity=0.05]
        (-2.5,-2.5) -- (-2.5,2.5)-- (2.5,2.5) --(2.5,-2.5) -- (-2.5,-2.5);
      \end{scope}
    \end{tikzpicture}
    \caption{A Hele--Shaw cell with spacing $b$}
    \label{fig:hele-shaw-cell-spacing}
  \end{subfigure}%
  \begin{subfigure}[t]{.5\linewidth}
    \centering
    \begin{tikzpicture}[scale=.8]
      \draw[thick, black, domain=0:2*pi,samples=200]
      plot ({deg( \x)}:{2-(cos(deg(3*\x)))^2});
      \draw[fill, blue, opacity=0.1,domain=0:2*pi,samples=200]
      plot ({deg( \x) }:{2-(cos(deg(3*\x)))^2});
	  \draw node at (0,0) {$\Omega(t)$};
      \draw[dashed, black] (-3,-3) -- (-3,3) -- (3,3) -- (3,-3) -- (-3,-3);
      \draw[fill, gray, opacity=0.05] (-3,-3) -- (-3,3) -- (3,3) -- (3,-3) --
      (-3,-3);
    \end{tikzpicture}
    \caption{Dimension reduction leads to the Hele--Shaw flow}
    \label{fig:hele-shaw-flow-illustration}
  \end{subfigure}
\end{figure}

	On the one hand, our result draws a connection between two well-known basic physical models. 
	On the other hand, the Cahn--Hilliard equation can be used as a numerical scheme to approximate solutions to the Hele--Shaw flow~\cite{Glasner2003}.
	In~\cite[Figure 2]{Glasner2003} a simulation shows
	the change of topology of a \emph{ferrofluid},
	which is subject to a constant external magnetic field perpendicular to the plates.
	The long, narrow droplet breaks up multiple times,
	eventually leading to an array of circular droplets.
	Our diffuse interface model can be extended to describe this experiment
	by adding a term to the free energy~\eqref{eq:def-energy-eps} describing the magnetic energy
	\[F_M(u)=2\pi M^2\int_{\mathbb{R}^d}u k_b*u\,dx,\]
	where $M$ is the magnetization and $k_b$ is a convolution kernel
	depending on the plate spacing $b$.
	For simplicity, we do not consider any additional terms in the energy in this
	work.

	The remainder of this work is structured as follows.
	In Section \ref{sec:heuristics}, we give the heuristic idea behind our strategy, which is then formalized in the following sections.
	In Section~\ref{sec:fix-eps-problem} we prove existence of weak solutions to the
	Cahn--Hilliard equation~\eqref{eq:cahn-hilliard-eq-fix-eps-classic}.
	In Section~\ref{sec:eps-limit} we investigate the sharp interface limit $\eps\downarrow 0$.
	Finally, in Appendix~\ref{sec:recovery-sequence} we construct well prepared
	initial data for a given initial configuration $\Omega(0)=\Omega_0$.

\section{Heuristic idea behind our proof} \label{sec:heuristics}

	In this work, we prove the following result; see Theorem~\ref{thm:u-eps-convergence-to-weak-hele-shaw-solution} for the precise statement.
	We consider (weak) solutions $u_\eps\geq 0$ of the Cahn--Hilliard equation, formally satisfying
	\begin{equation}
		\label{eq:cahn-hilliard-eq-fix-eps}
		\begin{cases}
			\partial_{t}u_{\eps}+\nabla\cdot j_\eps=0, \\
			j_{\eps}=u_\eps\nabla\left(\eps\Delta u_{\eps}-\frac{1}{\eps}W'(u_{\eps})\right),
		\end{cases}
	\end{equation}
	with given initial conditions $u_\eps(\cdot,0)=u_{\eps,0}.$ 
	Here $j_\eps$ is the flux
	and the mobility $m(u)=u$ degenerates in the phase $\{u=0\}$.
	We show that $u_\eps=u_\eps(x,t)$ converges to a characteristic function
	$\chi(x,t)=\chi_{\Omega(t)}(x)$ as the length scale $\eps$ vanishes and that, under an energy convergence assumption,
	there exists a flux field $j\in L^2(\mathbb{R}^d\times(0,T);\mathbb{R}^d))$ such that
	the pair $(\{\Omega(t)\}_{t\in[0,T]},j)$ is a weak solution to the Hele--Shaw
	flow
	\begin{align}
	\label{eq:velocity-eq}
	&\begin{cases}
	\nabla\cdot j(\cdot,t)=0,\quad&\text{in $\Omega(t)$,} \\
	V=j(\cdot,t)\cdot\nu,&\text{on $\partial\Omega(t)$,}
	\end{cases}
	\\
	\label{eq:conservation-eq}
	&\begin{cases}
	j(\cdot,t)=-\nabla p(\cdot,t),\quad&\text{in $\Omega(t)$,} \\
	\sigma H=p(\cdot,t),&\text{on $\partial\Omega(t)$.}
	\end{cases}
	\end{align}
	Here, $\sigma$ denotes the surface tension, $H$ denotes the mean curvature of the free boundary $\partial \Omega(t)$ and $\nu$ its normal vector. 
	In this sharp-interface model, the flux $j$ can be viewed as a fluid velocity, and $p$ plays the role of pressure.
	The first two equations~\eqref{eq:velocity-eq} state that the flow is incompressible and that the free boundary is transported by the fluid velocity, a simple kinematic condition.
	The second two equations~\eqref{eq:conservation-eq} are Darcy's law, which governs the slow motion of fluids in porous media or in narrow regions, and the force balance along the free boundary between capillary forces and pressure.
	
	In this section, we give the heuristic argument for this convergence. 
	To this end, let us assume we have a smooth solution to \eqref{eq:cahn-hilliard-eq-fix-eps}.
	A direct computation then shows that the Cahn--Hilliard energy \eqref{eq:def-energy-eps} is dissipated:
	\begin{align*}
		\frac{d}{dt} E_\eps(u_\eps) = - \int_{\mathbb{R}^d} \frac{|j_\eps|^2}{u_\eps} \,dx \leq 0.
	\end{align*}
	In particular, this means that
	\begin{align*}
		\sup_{t>0} E_\eps(u_\eps) \leq E_\eps(u_{\eps,0}) \quad \text{and}\quad \int_{0}^{\infty} \int_{\mathbb{R}^d} \frac{|j_\eps|^2}{u_\eps} \,dx\,dt \leq E_\eps(u_{\eps,0}).
	\end{align*}
	The first estimate gives compactness in configuration space, while the second estimate gives us control in time. 
	Indeed, by the Modica--Mortola/Bogomoln'yi trick \cite{Modica1977, Bogomolcprimenyi1976}, i.e., a combination of the chain rule and Young's inequality,
	\begin{align*}
		\int_{\mathbb{R}^d} \int |\nabla (\phi \circ u_\eps)(x,t)| \,dx \leq E_\eps(u_\eps).
	\end{align*}
	Using $v(\cdot,t'):=\big(\frac{ j_\eps}{u_\eps}\big)(\cdot, t' (t_2-t_1) +t_1)$ in the Benamou-Brenier formula for optimal transport and changing variables yields
	\begin{align*}
		W^2_2(u_\eps(\cdot,t_2), u_\eps(\cdot,t_1)) \leq (t_2-t_1) \int_{t_1}^{t_2} \int_{\mathbb{R}^d} \frac{|j_\eps|^2}{u_\eps} \,dx\,dt \leq E_\eps(u_{\eps,0}).
	\end{align*}
	Hence, it is natural to expect some compactness by some variant of the Aubin--Lions lemma.
	More precisely, we will show in Section~\ref{sec:sharp-interface-limit-compactness} that after passage to a subsequence
	\begin{align*}
		u_\eps \to \chi \quad \text{and} \quad j_\eps \stackrel{\ast}{\rightharpoonup} j 
	\end{align*}
	for some characteristic function $\chi$ and some flux $j$, satisfying   $\esssup_{t>0} \int |\nabla \chi| <\infty$ and 
	$\int_0^\infty \int_{\mathbb{R}^d} \chi |j|^2\,dx\,dt <\infty$.
	
	Now to verify that the limit $(\chi,j)$ satisfies the Hele--Shaw system~\eqref{eq:velocity-eq}--\eqref{eq:conservation-eq} in a distributional sense, we observe that the Cahn--Hilliard equation~\eqref{eq:cahn-hilliard-eq-fix-eps} has a divergence structure.
	This is clear for the first equation, but it is not immediate for the second one relating the flux $j_\eps$ to the first variation of the Cahn--Hilliard energy in Wasserstein space. 
	However, since here we assume that $u_\eps$ is smooth, some simple manipulations show that
	\begin{align}\label{eq:heuristic div T}
		u_\eps\nabla\left(\eps\Delta u_{\eps}-\frac{1}{\eps}W'(u_{\eps})\right) = - \nabla \cdot \boldsymbol{T}_{\eps} + 
		\nabla \left( \left(\eps|\nabla u_\eps|^2+\frac{1}{\eps}W'(u_\eps)u_\eps\right)
		-\nabla \cdot (u_\eps\nabla u_\eps)\right),
	\end{align}
	where $\boldsymbol{T}_{\eps}$ denotes the energy-stress tensor
	\begin{align*}
		\boldsymbol{T}_{\eps}=\left(\frac{\eps}{2}|\nabla u_{\eps}|^2+\frac{1}{\eps}W(u_{\eps})\right)I_d
		-\eps\nabla u_{\eps}\otimes\nabla u_{\eps},
	\end{align*}
	which appears naturally when performing domain variations.
	Let us comment on the identity \eqref{eq:heuristic div T}. 
	First, it is not surprising that also in our case of variations in Wasserstein space, the energy-stress tensor appears. 
	Second, when testing with a divergence-free vector field (in which case domain variations and variations in Wasserstein space are equivalent), the gradient-term vanishes and we recover the classical form known from domain variations.
	Third, the right-hand side of \eqref{eq:heuristic div T} equation is in divergence form, which means that the outermost derivative can be put onto the test function. 
	In addition, the last right-hand side term contains yet another divergence, which can be put onto the test function, too. 
	Then one ends up with only first-order operators on $u_\eps$. 
	This implies that the resulting weak formulation enjoys excellent compactness properties in the sense that given a sequence of weak solutions one only needs to show energy convergence (and does not need higher regularity) to prove that the limit is again a weak solution.
	Even better, \eqref{eq:heuristic div T} even allows us to pass to the sharp-interface limit $\eps\downarrow0$: we only need to pass to the limit in these first-order terms, which can be done by an adaptation of the seminal work of Luckhaus and Modica \cite{Luckhaus1989}, who in particular prove that if $u_\eps \to \chi$ such that $E_\eps(u_\eps) \to \sigma \int|\nabla \chi|$, then $\boldsymbol{T}_\eps \stackrel{\ast}{\rightharpoonup} \boldsymbol{T} = \left( I_d - \nu \otimes \nu\right) |\nabla \chi|$, where $\nu $ is the measure theoretic inner unit normal given by the Radon--Nikodym derivative $\frac{\nabla \chi}{|\nabla \chi|}$.
	
	In order to turn this idea into a rigorous proof, we first construct weak solutions $u_\eps$ in Section~\ref{sec:fix-eps-problem}, which will already be based on the divergence structure \eqref{eq:heuristic div T}. 
	Then, in Section~\ref{sec:sharp-interface-limit-compactness}, we make rigorous the compactness, and in Section~\ref{sec:sharp-interface-limit-convergence}, we pass to the limit $\eps\downarrow0$ in our weak formulation.
	
\section{Construction of weak solutions to the Cahn--Hilliard equation}\label{sec:fix-eps-problem}

The main result of this section is the following theorem on global-in-time
existence of weak solutions to the degenerate Cahn--Hilliard
equation~\eqref{eq:cahn-hilliard-eq-fix-eps-classic}.

\begin{theorem}
  \label{thm:existence-sol-to-cahn-hilliard-eq-fix-eps}
  For each $u_{\eps,0}\in\mathcal{A}$ and each $\eps>0$,
  there exists a weak solution to the Cahn--Hilliard equation
  in the sense of Definition~\ref{def:weak-sol-cahn-hilliard} below.
\end{theorem}

Here $\mathcal{A}\subset L^{1}(\mathbb{R}^{d};[0,\infty))$ denotes
the set of nonnegative probability densities $u$
with respect to the Lebesgue measure,
which is denoted by $\mathcal{L}^d$, with finite second moments, i.e.,
\begin{align*}
  \mathcal{A}\coloneqq\{
  u\in L^{1}(\mathbb{R}^{d};[0,\infty)):
  \int_{\mathbb{R}^d}u\,dx=1,\,
  M_2(u)\coloneqq\int_{\mathbb{R}^d}|x|^2 u(x)\,dx <\infty
  \}.
\end{align*}
To define our weak solutions, we use the weak formulation \eqref{eq:cahn-hilliard-eq-fix-eps}, which is formulated in terms of the field $u_\eps$ and the flux $j_\eps$.

\begin{definition}[Weak solution to Cahn--Hilliard]
  \label{def:weak-sol-cahn-hilliard}
  Let $\eps>0$ and let $u_{\eps,0}\in\mathcal{A}$.
  We say that $(u_\eps,j_\eps)$ is a
  \emph{weak solution to the Cahn--Hilliard equation~\eqref{eq:cahn-hilliard-eq-fix-eps}}, if
  \begin{enumerate}[(i)]
  \item $u_\eps$ and $j_\eps$ are a distributional solution to~\eqref{eq:cahn-hilliard-eq-fix-eps}, i.e.,
    $u_\eps(\cdot,0)=u_{\eps,0}$, and
    \begin{equation}
      \label{eq:weak-sol-cahn-hilliard-eq-1}
      \int_{\mathbb{R}^d}u_{\eps,0}\zeta(\cdot,0)\,dx
      +\int_0^T\int_{\mathbb{R}^d}u_\eps\partial_t\zeta
      +j_\eps\cdot\nabla\zeta\,dx\,dt=0
    \end{equation}
    holds for all $\zeta\in C_c^2(\mathbb{R}^d\times[0,T))$, and
    \begin{equation}
      \label{eq:weak-sol-cahn-hilliard-eq-2}
      \begin{split}
        &\int_0^T\int_{\mathbb{R}^d}j_\eps\cdot\xi\,dx\,dt \\
        =\,\,&\int_0^T\int_{\mathbb{R}^d}\boldsymbol{T}_\eps:\nabla\xi\,dx\,dt \\
        &-\int_0^T\int_{\mathbb{R}^d}
        \left[
          (\nabla\cdot\xi)\left(\eps|\nabla u_\eps|^2+\frac{1}{\eps}W'(u_\eps)u_\eps\right)
          +u_\eps\nabla u_\eps\cdot\nabla(\nabla\cdot\xi)
        \right]\,dx\,dt
      \end{split}
    \end{equation}
    for all $\xi\in C_c^1(\mathbb{R}^d\times(0,T);\mathbb{R}^d)$,
    where $\boldsymbol{T}_\eps$ is the \emph{energy stress tensor}
    \begin{equation}
      \label{eq:def-T-eps}
      \boldsymbol{T}_{\eps}\coloneqq\left(\frac{\eps}{2}|\nabla u_{\eps}|^2+\frac{1}{\eps}W(u_{\eps})\right)I_d
      -\eps\nabla u_{\eps}\otimes\nabla u_{\eps}.
    \end{equation}
    Here $I_d$ denotes the identity matrix on $\mathbb{R}^d$.
  \item The solution satisfies the optimal energy-dissipation rate
    \begin{equation}
      \label{eq:optimal-energy-dissipation}
      \esssup_{T'\in[0,T]}\{E_{\eps}(u_{\eps}(\cdot,T'))
      +\int_{0}^{T'}\int_{\mathbb{R}^d}\frac{|j_\eps|^2}{u_\eps}\,dx\,dt\}
      \leq E_{\eps}(u_{\eps,0}).
    \end{equation}
  \end{enumerate}
\end{definition}

Here, with a slight abuse of notation, we work with the convention $\frac{0}{0}=0$.
In particular we see that, for a.e.\ $x,t$, if $u_\eps(x,t)=0$, then $j_\eps(x,t)=0$.

Elliot and Garcke~\cite{Elliott1996} have used a Galerkin approximation approach to construct weak solutions
to Cahn--Hilliard equations with mobility functions $m\geq\delta$, and then take
$\delta\downarrow 0$.
To be self-contained, we give an alternative existence proof of weak solutions
to~\eqref{eq:cahn-hilliard-eq-fix-eps}.
We utilize a minimizing movements scheme directly for the degenerate case
$m(u)=u$ and exploit the connection to optimal transport as discovered by
Jordan, Kinderlehrer, and Otto~\cite{Jordan1998}.
Here our energy is of higher order, but non-negative.

Without loss of generality, by scaling, we assume $\eps=1$ and drop the index $\eps$ in this section for
notational simplicity.

\subsection{The construction}

Let $T\in(0,\infty)$.
For $h>0$ we consider the minimization problem
\begin{equation}
  \label{eq:min-problem-uh}
  \inf_{u\in L^{2}}\{ E(u)+\frac{1}{2h}d^{2}(u,u_{0})\},
\end{equation}
where $u_{0}\in\mathcal{A}$ such that $E(u_{0})<\infty$.
Here $d(u,u_0)\coloneqq W_2(u\mathcal{L}^d,u_0\mathcal{L}^d)$ denotes the
quadratic Wasserstein distance of densities, and the Wasserstein distance of
measures $\mu,\nu$ is given by~\eqref{eq:kantorovich-formulation}.
We refer to Appendix~\ref{sec:recap-on-optimal-transport} for some basic facts
and the standard notation we use here.

\begin{lemma}[Existence of minimizers]
  \label{lem:existence-of-minimizers-mm-fix-eps}
  Let $u_{0}\in\mathcal{A}$ such that $E(u_0)<\infty$ and let $h>0$.
  Then there exists $u\in\mathcal{A}$ which minimizes~\eqref{eq:min-problem-uh}.
\end{lemma}

We use this lemma inductively to construct a sequence
$(u_{n})_{n\in\mathbb{N}}$ such that
\begin{equation}
  \label{eq:mm-scheme}
  u_n\in\argmin_u\{E(u)+\frac{1}{2h}d^{2}(u,u_{n-1})\},
  \quad n=1,2,3,\dots
\end{equation}
Then we define the approximation $u_h:[0,\infty)\times\mathbb{R}^d\rightarrow\mathbb{R}$
by piecewise constant interpolation using
$u_{h}(x,t)=u_{n}(x)$ for $t\in[hn,h(n+1))$.
Finally, we show that $u_{h}$ is precompact in
$L^1(0,T;L^1(\mathbb{R}^{d}))$.

This variational algorithm is known as the minimizing movement/JKO scheme,
and was first introduced in~\cite{Jordan1998}.

\setcounter{step}{0}
\begin{proof}[Proof of Lemma~\ref{lem:existence-of-minimizers-mm-fix-eps}]
  We use the direct method to prove existence of a minimizer for \eqref{eq:min-problem-uh}.

  Since $E(u_{0})+\frac{1}{2h}d^{2}(u_{0},u_{0})=E(u_{0})<\infty$, the
  infimum~\eqref{eq:min-problem-uh} is bounded from above.
  It is also bounded from below, since $E(u)+\frac{1}{2h}d^2(u,u_0)$ is
  nonnegative for all $u$.
  Note that both functionals
  $u\mapsto E(u)$ and $u\mapsto d^2(u,u_0)$ are lower semi-continuous w.r.t.
  $L^1$ convergence.
  Hence we only need to show compactness.
  
  Let $(u_{l})_{l\geq 1}\subset\mathcal{A}$ be a minimizing sequence for $E(\cdot)+\frac{1}{2h}d^{2}(\cdot,
  u_{0})$;
  note that if $u\notin\mathcal{A}$, then $d^{2}(u,u_{0})=\infty$.
  For all sufficiently large $l$ we then have
  \begin{equation}
    \label{eq:seq-energy-bound}
    E(u_{l})+\frac{1}{2h}d^{2}(u_{l},u_{0})\leq E(u_{0})+1<\infty,
  \end{equation}
  so we may assume w.l.o.g.\ that~\eqref{eq:seq-energy-bound} holds for all $l$.
  Now we observe that there exists $M<\infty$ such that
  \begin{equation}
    \label{eq:W-order-u4}
    W(u)=\frac{1}{4}u^{2}(u-1)^{2}\geq\frac{1}{5}u^{4}
    \quad\text{for all $|u|>M$.}
  \end{equation}
  Now let $u\in\mathcal{A}\cap L^4(\mathbb{R}^d)$.
  Using~\eqref{eq:W-order-u4} we have
  \begin{align*}
    \int_{\mathbb{R}^{d}}u^{4}\,dx
    &=\int_{\{x:u(x)\leq M\}}u^{4}\,dx+\int_{\{x:u(x)>M\}}u^{4}\,dx \\
    &\leq M^{3}\int_{\mathbb{R}^{d}} u\,dx+5\int_{\{x:u(x)>M\}} W(u)\,dx \\
    &\leq M^{3}+5E(u),
  \end{align*}
  and plugging in $u=u_{l}$ we get by~\eqref{eq:seq-energy-bound}
  \begin{equation}
    \label{eq:9}
    \int_{\mathbb{R}^d} u_{l}^{4}\,dx\leq M^{3}+5E(u_{0})+5.
  \end{equation}
  Thus $u_{l}$ is uniformly bounded in $L^{4}$.
  Further we have by~\eqref{eq:seq-energy-bound}
  \[\int_{\mathbb{R}^{d}}|\nabla u_{l}|^{2}\,dx\leq 2E(u_{0})+2,\]
  hence $\nabla u_{l}$ is also uniformly bounded in $L^{2}$.
  By Rellich's Theorem~\cite[Chap.\ 5.7, Thm.\ 1]{Evans2010} and a diagonal argument,
  there exists a subsequence $u_{l_{m}}\rightarrow u$ converging in $L^{2}_{loc}$.

  To obtain $L^1$ convergence, it suffices to show that the second moments are
  uniformly bounded and $L^1_{loc}$ convergence;
  Indeed, if the seconds moments are uniformly bounded from above, we have
  \begin{align*}
    M_{2}(u)=\lim_{R\rightarrow\infty}\int_{B_{R}}|x|^{2}u_{l_m}(x)\,dx
    &\leq\lim_{R\rightarrow\infty}\limsup_{m\rightarrow\infty}\int_{B_{R}}|x|^{2}u_{l_m}(x)\,dx \\
    &\leq\limsup_{m\rightarrow\infty}\int_{\mathbb{R}^{d}}|x|^{2}u_{l_m}(x)\,dx<\infty.
  \end{align*}
  Then, for all $R<\infty$ we have
  \begin{align*}
    \int_{\mathbb{R}^{d}}|u_{l_m}-u|\,dx
    &=\int_{B_{R}}|u_{l_m}-u|\,dx+\int_{\mathbb{R}^{d}\setminus B_{R}}|u_{l_m}-u|\,dx \\
    &\leq\int_{B_{R}}|u_{l_m}-u|+\frac{1}{R^{2}}(M_{2}(u_{l_m})+M_{2}(u)).
  \end{align*}
  Taking the limit $m\rightarrow\infty$ we get
  \[\limsup_{m\rightarrow\infty}\int_{\mathbb{R}^{d}}|u_{l_m}-u|
    \leq\frac{1}{R^{2}}\left(\sup_{m}M_{2}(u_{l_m})+M_{2}(u)\right),
  \]
  and taking $R\rightarrow\infty$ on the RHS, we obtain
  \[\limsup_{m\rightarrow\infty}\int_{\mathbb{R}^{d}}|u_{l_m}-u|\,dx=0.\]
  Now we show that the second moments are uniformly bounded.
  To this end, let $\gamma_{l}$ be the optimal plan in
  the optimal transport problem~\eqref{eq:kantorovich-formulation}, i.e.,
  \[(\pi_{x})_{\sharp}\gamma_{l}=u_{l}\mathcal{L}^{d},\quad
    (\pi_{y})_{\sharp}\gamma_{l}=u_{0}\mathcal{L}^{d},\quad
    d^{2}(u_{l},u_{0})=\int_{\mathbb{R}^{d}\times\mathbb{R}^{d}}|x-y|^{2}d\gamma_{l}(x,y).
  \]
  We integrate the inequality
  \begin{equation}
    \label{eq:train-ineq-young}
    |x|^{2}
    =|x-y+y|^2
    =|x-y|^{2}+2(x-y)\cdot y+|y|^{2}
    \overset{\text{Young}}\leq\left(1+\frac{1}{h}\right)|x-y|^{2}+(1+h)|y|^{2}
  \end{equation}
  against the optimal plan $\gamma_{l}$ to get
  \begin{equation}
    \label{eq:uniform-moment-bound-in-terms-of-prev-energy-and-moment}
    \begin{split}
      M_{2}(u_{l})
      &=\int_{\mathbb{R}^{d}}|x|^{2}u_{l}(x)\,dx \\
      &\leq\left(
        1+\frac{1}{h}
      \right)d^2(u_{l},u_{l-1})
      +(1+h)\int_{\mathbb{R}^{d}}|y|^{2}u_{0}(y)\,dy \\
      &\leq 2(1+h)(E(u_{0})-E(u_{l}))+(1+h)M_{2}(u_{0}) \\
      &\leq 2(1+h)E(u_{0})+(1+h)M_{2}(u_{0}).
    \end{split}
  \end{equation}
  The RHS is independent of $l$, so we have a uniform bound on the second moments of $u_{l}$.
\end{proof}

Now we define $u_{n},\,n=1,2,3,\dots$ successively as a minimizer of
\[\inf_{u\in\mathcal{A}}E(u)+\frac{1}{2h}d^{2}(u,u_{n-1}).\]
The only assumption on $u_0$ in
Lemma~\ref{lem:existence-of-minimizers-mm-fix-eps} was that
$u_{0}\in\mathcal{A}$ with $E(u_0)<\infty$,
and any minimizer of $E(\cdot)+\frac{1}{2h}d^{2}(u_{0},\cdot)$ is again in $\mathcal{A}$.
Thus it is guaranteed that a $u_{n}$ exists for all $n\in\mathbb{N}$.
For $h>0$ define the piecewise constant time interpolation
$u_h:\mathbb{R}^d\times[0,\infty)\rightarrow\mathbb{R}$ by
\begin{equation}
  \label{eq:def-uh}
  u_{h}(x,t)\coloneqq u_{n}(x),\quad t\in[nh,(n+1)h).
\end{equation}

Then we have a uniform bound on the energy of $u_n$.
The next lemma gives us a bound on the second moments of $u_n$,
which we need to prove compactness of $u_h$.

\begin{lemma}
  \label{lem:uniform-moment-bound}
  Let $u_n\in\mathcal{A}$ be a sequence and $h\in(0,1)$ such that
  \[E(u_{n})+\frac{1}{2h}d^{2}(u_{n},u_{n-1})\leq E(u_{n-1})\quad
    \text{for all $n=1,2,3,\dots$}
  \]
  Then the second moments $M_2(u_n)$ satisfy the estimate
  \begin{align*}
    M_{2}(u_{n})\leq Ce^{Cnh}M_{2}(u_{0})
    \quad\text{for all $n=1,2,3,\dots$}
  \end{align*}
\end{lemma}

\begin{proof}
  As in the proof of Lemma~\ref{lem:existence-of-minimizers-mm-fix-eps},
  we use~\eqref{eq:train-ineq-young}
  and~\eqref{eq:uniform-moment-bound-in-terms-of-prev-energy-and-moment}, except
  instead of the optimal plan between $u_n$ and $u_0$, we use the optimal plan
  between $u_n$ and $u_{n-1}$.
  Then
  \begin{align*}
    M_{2}(u_{n})\leq 2(1+h)(E(u_{n-1})-E(u_{n}))+(1+h)M_{2}(u_{n-1}).
  \end{align*}
  This is equivalent to the inequality
  \[\frac{M_{2}(u_{n})-M_{2}(u_{n-1})}{h}
    \leq 2(1+h)\frac{E(u_{n-1})-E(u_{n})}{h}+M_{2}(u_{n-1}),
  \]
  which we can rewrite as
  \begin{align*}
    \frac{(M_{2}(u_{n})+2E(u_{n}))-(M_{2}(u_{n-1})+2E(u_{n-1}))}{h}
    &\leq M_{2}(u_{n-1})+2(E(u_{n-1})-E(u_{n})) \\
    &\leq M_{2}(u_{n-1})+2(E(u_{n-1}).
  \end{align*}
  Now Gronwall's inequality~\cite[Prop.\ 3.1]{emmrich1999discrete} yields
  \begin{align*}
    M_{2}(u_{n})\leq Ce^{Cnh}M_{2}(u_{0})
  \end{align*}
  for all $h\in(0,1)$.
\end{proof}

Now we are ready to prove that the piecewise constant time interpolation is
precompact in $L^1$.

\begin{lemma}[Compactness]
  \label{lem:mm_compactness_fix_eps_problem}
  For all finite time horizons $T<\infty$, there exists a subsequence $h_{l}\rightarrow 0$ as
  $l\rightarrow\infty$ such that $u_{h_{l}}\rightarrow u$ for some $u$ in
  $L^1(0,T;L^{p}(\mathbb{R}^d))$ for all $1\leq p\leq 4$.
\end{lemma}

Of course, this lemma can be extended to
$L^1_{loc}([0,\infty);L^1(\mathbb{R}^d))$ convergence using a diagonal argument.

\setcounter{step}{0}
\begin{proof}
  The proof is divided into two steps.
  First, we show convergence in $L^1_{loc}$, and then post-process it to $L^1$ convergence.
  \begin{step}[$L^1_{loc}$ convergence]
    We have
    \begin{equation}
      \label{eq:10}
      E(u_{n})+\frac{1}{2h}d^{2}(u_{n},u_{n-1})
      \leq E(u_{n-1}),
    \end{equation}
    since $u=u_{n-1}$ is admissible in the minimization problem for $u_{n}$.
    Then by induction,
    \begin{equation}
      \label{eq:11}
      E(u_{n})
      +\frac{h}{2}\sum_{l=1}^{n}\left(
        \frac{d(u_{l},u_{l-1})}{h}
      \right)^{2}
      \leq E(u_{0}).
    \end{equation}
    Thus $\nabla u_{n}$ is uniformly bounded in $L^{2}$ and
    \begin{equation}
      \label{eq:discrete-energy-bounded}
      E(u_h(\cdot,T))+\frac{1}{2}\int_0^{T-h}\left(\frac{d(u_h(\cdot,t),u_h(\cdot,t+h))}{h}\right)^2\,dt
      \leq E(u_{0}).
    \end{equation}
    Further, as in~\eqref{eq:W-order-u4}--\eqref{eq:9}, $u_{n}$ is uniformly bounded in $L^{4}$.
    Then there exists a constant $C<\infty$ such that, for any $T>0$, we have
    \begin{align*}
      \int_{0}^{T}\int_{\mathbb{R}^{d}}|u_{h}(x,t)|^{4}\,dx\,dt
      &\leq CTE(u_0), \\
      \int_{0}^{T}\int_{\mathbb{R}^{d}}|\nabla u_{h}(x,t)|^{2}\,dx\,dt
      &\leq 2TE(u_0), \\
      \frac{1}{2}\int_0^{T-h}\left(\frac{d(u_h(\cdot,t),u_h(\cdot,t+h))}{h}\right)^2\,dt
      &\leq E(u_0).
    \end{align*}
    Then $u_h\in L^2(0,T;L^2(\mathbb{R}^d))$ and
    $\nabla u_h\in L^2(0,T;L^2(\mathbb{R}^d;\mathbb{R}^d))$ for all $h$
    with uniform bounds.
    
    Now let $r<\infty$.
    By the Aubin-Lions Lemma~\cite[Thm.\ 1]{Aubin1963}
    (see also~\cite[Prop.\ 8]{Chen2014}),
    there exists a subsequence $u_{h}^r$ converging to some $u^r$ in
    $L^2(B_r(0)\times [0,T])$.
    For $s>r$ we can again find a subsequence $(u_h^s)_h\subset(u_h^r)_h$
    and $u^s$ such that $u_h^s\rightarrow u^s$ in $L^2(B_s(0)\times[0,T])$.
    Then
    \[\|u^s-u^r\|_{L^2(B_r(0)\times[0,T])}
      \leq\|u^s-u_h^s\|_{L^2(B_r(0)\times[0,T])}
      +\|u_h^s-u^r\|_{L^2(B_r(0)\times[0,T])}
      \rightarrow 0,
    \]
    because $u_h^s$ is a subsequence of $u_h^r$.
    Thus $u^s=u^r$ a.e.\ on $B_r(0)\times[0,T]$.
    Iterating this process, we obtain a well defined function
    $u\in L^2_{loc}(\mathbb{R}^d\times[0,T])$ and a subsequence $u_{h_l}$, e.g.\
    the diagonal sequence, such that $u_{h_l}\rightarrow u$ in $L^2(B_r(0)\times[0,T])$
    for all $r<\infty$, hence $u_{h_l}\rightarrow u$ in
    $L^2_{loc}(\mathbb{R}^d\times[0,T])$, and also in $L^1_{loc}(\mathbb{R}^d\times[0,T])$.
  \end{step}
  \begin{step}[$L^1$ convergence]
    By Lemma~\ref{lem:uniform-moment-bound} we know that
    \begin{equation}
      \label{eq:second-moment-bound}
      M_{2}(u_{n})\leq Ce^{Cnh}M_{2}(u_{0})
    \end{equation}
    holds for all $h\in(0,1)$.
    Now observe that for any $R<\infty$ we have
    \begin{align*}
      \int_{0}^{T}\int_{\mathbb{R}^{d}}|u-u_{h_{l}}|\,dx\,dt
      &\leq\int_{0}^{T}\int_{B_{R}}|u-u_{h_{l}}|\,dx\,dt
        +\frac{1}{R^{2}}\int_{0}^{T}\int_{\mathbb{R}^{d}}|x|^{2}(|u|+|u_{h_{l}}|)\,dx\,dt \\
      &\leq\int_{0}^{T}\int_{B_{R}}|u-u_{h_{l}}|\,dx\,dt+\frac{CT}{R^{2}},
    \end{align*}
    where $C=\sup_{t\in[0,T]}\{M_{2}(u(\cdot,t))+M_{2}(u_{h_{l}}(\cdot,t))\}$,
    which is finite by~\eqref{eq:second-moment-bound}.
    Since $u_{h_{l}}\rightarrow u$ in
    $L^{1}_{loc}(\mathbb{R}^d\times[0,T])$, taking first
    $h_{l}\rightarrow 0$ and then $R\rightarrow\infty$, we get strong
    convergence in $L^{1}(0,T;L^1(\mathbb{R}^d))$.

    Since $u_{h_l}$ is uniformly bounded in $L^4$, we have $u\in L^4$.
    By interpolation between $L^p$ norms~\cite[Appendix B.2.h]{Evans2010},
    we get $u_{h_l}\rightarrow u$ in $L^1(0,T;L^p(\mathbb{R}^d))$
    for all $1\leq p\leq 4$. \qedhere
  \end{step}
\end{proof}

\begin{remark}\label{rem:uh-uniformly-bounded-in-H1}
  We remark that this proof also shows that
  $u_h$ is uniformly bounded in $L^2(0,T;H^1(\mathbb{R}^d))$ as $h\downarrow 0$
  and that $u_h\rightarrow u$ weakly in $L^2(0,T;H^1(\mathbb{R}^d))$.
  Also the second moments are uniformly bounded, i.e.\ there exists $C<\infty$
  depending only on $E(u_0)$ and $M_2(u_0)$ such that for all $n$
  \begin{equation}
    \label{eq:u-n-second-moments-uniformly-bounded}
    M_2(u_n)\leq C.
  \end{equation}
\end{remark}

The next step will be to compute the Euler-Lagrange equation.
Before we do this, we compute the first variation of the
Dirichlet energy.
For the remainder of this paper,
$\boldsymbol{i}_d$ denotes the identity map on $\mathbb{R}^d$.

\begin{lemma}
  \label{lem:variation-dirichlet-energy}
  Let $u_0\in H^1(\mathbb{R}^d)$ and let
  $\xi\in C_c^\infty(\mathbb{R}^d;\mathbb{R}^d)$.
  Further, let $\{\boldsymbol{t}_t\}_{t\in\mathbb{R}}$ be the corresponding flux
  \[
    \begin{cases}
      \partial_t\boldsymbol{t}_t=\xi\circ\boldsymbol{t}_t, \\
      \boldsymbol{t}_0=\boldsymbol{i}_d,
    \end{cases}
  \]
  and let $u_t=(\boldsymbol{t}_t)_\sharp u_0$ be the push-forward of $u_0$ under
  $\boldsymbol{t}_t$.
  Then
  \begin{equation}
    \label{eq:first-var-dirichlet-energy}
    \frac{d}{dt}\bigg|_{t=0}\int_{\mathbb{R}^d}
    \frac{1}{2}|\nabla u_t|^2\,dx
    =\int_{\mathbb{R}^d}
    -\nabla u_0\cdot\nabla\xi\nabla u_0
    -\frac{1}{2}|\nabla u_0|^2\nabla\cdot\xi
    +u_0\nabla u_0\cdot\nabla(\nabla\cdot\xi)\,dx.
  \end{equation}
\end{lemma}

Of course, if $u_0\in H^2$, the computation
for~\eqref{eq:first-var-dirichlet-energy} is straight forward using the
continuity equation $\partial_t u+\nabla\cdot(u\xi)=0$.
In a slightly different very general setting~\cite{Lisini2012}, it was shown
that each solution $u_n$ to the minimizing movement scheme~\eqref{eq:mm-scheme}
satisfies $u_n\in H^2$.
Also Elliott and Garcke~\cite{Elliott1996} prove a similar result.

However, as we will see, $H^1$ is enough regularity to
verify~\eqref{eq:first-var-dirichlet-energy}.
To the best of our knowledge, this has not yet been done in the present setting.

\begin{proof}
  By definition of $u_t$, we have
  \[\int_{\mathbb{R}^d}u_t\zeta\,dx
    =\int_{\mathbb{R}^d}u_0(x)\zeta(\boldsymbol{t}_t(x))\,dx\quad
    \text{for all $\zeta\in C_c^0(\mathbb{R}^d)$.}
  \]
  Since $\boldsymbol{t}_t$ is invertible, the above is equivalent to
  \begin{equation}
    \label{eq:un-is-det-ut-circ-tt}
    u_0=\det(\nabla\boldsymbol{t}_t)u_t\circ\boldsymbol{t}_t,
  \end{equation}
  and using $\partial_t\boldsymbol{t}_t=\xi\circ\boldsymbol{t}_t$, we have
  \begin{equation}
    \label{eq:tt-properties}
    \begin{cases}
      \nabla\boldsymbol{t}_t|_{t=0}=I_d, \\
      \partial_t\nabla\boldsymbol{t}_t|_{t=0}=\nabla\xi, \\
      \partial_t\det\nabla\boldsymbol{t}_t|_{t=0}
      =\tr\nabla\xi
      =\nabla\cdot\xi.
    \end{cases}
  \end{equation}
  First, we compute the gradient of~\eqref{eq:un-is-det-ut-circ-tt}
  \begin{equation}
    \label{eq:un-is-det-ut-circ-tt-grad}
    \nabla u_0
    =u_t\circ\boldsymbol{t}_t\nabla\det\nabla\boldsymbol{t}_t
    +(\det\nabla\boldsymbol{t}_t)(\nabla\boldsymbol{t}_t)^T(\nabla u_t)\circ\boldsymbol{t}_t.
  \end{equation}
  By the Jacobi formula for the gradient of the determinant, the first term reads
  \[\partial_{x_i}\det\nabla\boldsymbol{t}_t
    =\det(\nabla\boldsymbol{t}_t)\tr((\nabla\boldsymbol{t}_t)^{-1}\partial_{x_i}\nabla\boldsymbol{t}_t).
  \]
  Therefore
  \begin{equation}
    \label{eq:gradient-of-det}
    \begin{split}
      u_t\circ\boldsymbol{t}_t
      \nabla\det\nabla\boldsymbol{t}_t
      &=u_t\circ\boldsymbol{t}_t\det(\nabla\boldsymbol{t}_t)
      \left( \tr((\nabla\boldsymbol{t}_t)^{-1}\partial_{x_i}\nabla\boldsymbol{t}_t) \right)_{i=1}^d \\
      &=u_0\left( \tr((\nabla\boldsymbol{t}_t)^{-1}\partial_{x_i}\nabla\boldsymbol{t}_t) \right)_{i=1}^d.
    \end{split}
  \end{equation}
  Here $v=(v_i)_{i=1}^d$ denotes the vector $v\in\mathbb{R}^d$ with components $v_i$.
  Rearranging terms in~\eqref{eq:un-is-det-ut-circ-tt-grad} and inserting~\eqref{eq:gradient-of-det} gives us
  \[(\nabla u_t)\circ\boldsymbol{t}_t
    =\frac{1}{\det\nabla\boldsymbol{t}_t}(\nabla\boldsymbol{t}_t)^{-T}\left(
      \nabla u_0-u_0\left(
        \tr((\nabla\boldsymbol{t}_t)^{-1}\partial_{x_i}\nabla\boldsymbol{t}_t)
      \right)_{i=1}^d
    \right).
  \]
  Therefore
  \begin{align*}
    \int_{\mathbb{R}^d}\frac{1}{2}|\nabla u_t|^2\,dx
    &=\int_{\mathbb{R}^d}\frac{1}{2}|(\nabla u_t)\circ\boldsymbol{t}_t|^2
      \det\nabla\boldsymbol{t}_t\,dx \\
    &=\int_{\mathbb{R}^d}\frac{1}{2}
      \frac{1}{\det\nabla\boldsymbol{t}_t}
      \left| \left( \nabla\boldsymbol{t}_t \right)^{-T}
      \left( \partial_{x_i}u_0
      -u_0\tr((\nabla\boldsymbol{t}_t)^{-1}\partial_{x_i}\nabla\boldsymbol{t}_t)
      \right)_{i=1}^d
      \right|^2\,dx.
  \end{align*}
  Now we can write the difference quotient as
  \begin{equation}
    \label{eq:variation-dirichlet-energy-difference-quotient}
    \begin{split}
      &\frac{1}{t}\left(
        \int_{\mathbb{R}^d}\frac{1}{2}|\nabla u_t|^2\,dx
        -\int_{\mathbb{R}^d}\frac{1}{2}|\nabla u_0|^2\,dx \right) \\
      =\,\,&\frac{1}{t}\left(
        \int_{\mathbb{R}^d}\frac{1}{2}|(\nabla u_t)\circ\boldsymbol{t}_t|^2
        \det\nabla\boldsymbol{t}_t
        \,dx
        -\int_{\mathbb{R}^d}\frac{1}{2}|\nabla u_0|^2\,dx \right) \\
      =\,\,&\int_{\mathbb{R}^d}\frac{1}{2}\delta_t\left\lbrace
        \frac{1}{\det\nabla\boldsymbol{t}_t}
        \left| \left( \nabla\boldsymbol{t}_t \right)^{-T}
          \left( \partial_{x_i}u_0
            -u_0\tr((\nabla\boldsymbol{t}_t)^{-1}\partial_{x_i}\nabla\boldsymbol{t}_t)
          \right)_{i=1}^d
        \right|^2\right\rbrace\,dx,
    \end{split}
  \end{equation}
  where $\delta_t$ denotes the difference quotient at zero, i.e.,
  $\delta_t(f(t))=\frac{f(t)-f(0)}{t}$.
  Now we want to pull the limit inside the integral by using dominated convergence.
  Let $v_k\coloneqq\tr((\nabla\boldsymbol{t}_t)^{-1}\partial_{x_k}\nabla\boldsymbol{t}_t)$.
  We write out all the scalars in~\eqref{eq:variation-dirichlet-energy-difference-quotient}
  to obtain a dominating function:
  \begin{equation}
    \label{eq:compute-difference-quotient}
    \begin{split}
      &\delta_{t}\left( \frac{1}{2\det\nabla\boldsymbol{t}_t}
        \sum_{i,j,k}
        \left[(\nabla\boldsymbol{t}_t)^{-T}\right]_{ij}\left[(\nabla\boldsymbol{t}_t)^{-T}\right]_{ik}
        \left[
          \partial_{x_j}u_0
          -u_{0}v_j
        \right]\left[
          \partial_{x_k}u_0
          -u_{0}v_k
        \right]
      \right) \\
      =\,\,&\delta_{t}\left( \frac{1}{2\det\nabla\boldsymbol{t}_t}
        \sum_{i,j,k}
        \left[(\nabla\boldsymbol{t}_t)^{-T}\right]_{ij}\left[(\nabla\boldsymbol{t}_t)^{-T}\right]_{ik}
        \left[
          (\partial_{x_j}u_0)(\partial_{x_k}u_0)
          -2u_0(\partial_{x_j}u_0)v_k
          +u_0^2 v_j v_k
        \right]
      \right).
    \end{split}
  \end{equation}
  We compute the time derivatives for all terms:
  \begin{align*}
    \partial_t(\nabla\boldsymbol{t}_t)^{-T}\big|_{t=0}
    &=-(\nabla\boldsymbol{t}_t)^{-T}(\partial_t\nabla\boldsymbol{t}_t)(\nabla\boldsymbol{t}_t)^{-T}\big|_{t=0} \\
    &=-\nabla\xi, \\
    \partial_t v_k
    &=\tr((\nabla\boldsymbol{t}_t)^{-1}\partial_{x_k}\nabla\partial_t\boldsymbol{t}_t)\big|_{t=0} \\
    &=\tr(\partial_{x_k}\nabla\xi)
      =\partial_k(\nabla\cdot\xi), \\
    \partial_t\big|_{t=0}\frac{1}{\det\nabla\boldsymbol{t}_t}
    &=-\frac{1}{(\det\nabla\boldsymbol{t}_t)^2}\partial_t\det\nabla\boldsymbol{t}_t \\
    &=-\nabla\cdot\xi.
  \end{align*}
  Since $\xi$ has compact support, all terms are bounded in $L^\infty$.
  Now we split the difference quotient~\eqref{eq:compute-difference-quotient} into three terms:
  \begin{align*}
    &\delta_{t}\left( \frac{1}{2\det\nabla\boldsymbol{t}_t}
      \sum_{i,j,k}
      \left[(\nabla\boldsymbol{t}_t)^{-T}\right]_{ij}\left[(\nabla\boldsymbol{t}_t)^{-T}\right]_{ik}
      \left[
      (\partial_{x_j}u_0)(\partial_{x_k}u_0)
      -2u_0(\partial_{x_j}u_0)v_k
      +u_0^2 v_j v_k
      \right]
      \right) \\
    =\,\,&\sum_{i,j,k}
           \delta_{t}\left( \frac{1}{2\det\nabla\boldsymbol{t}_t}
           \left[(\nabla\boldsymbol{t}_t)^{-T}\right]_{ij}\left[(\nabla\boldsymbol{t}_t)^{-T}\right]_{ik}
           \right)(\partial_{x_j}u_0)(\partial_{x_k}u_0) \\
    &-2\sum_{i,j,k}
      \delta_{t}\left( \frac{1}{2\det\nabla\boldsymbol{t}_t}
      \left[(\nabla\boldsymbol{t}_t)^{-T}\right]_{ij}\left[(\nabla\boldsymbol{t}_t)^{-T}\right]_{ik}v_k
      \right)u_0(\partial_{x_j}u_0) \\
    &+\sum_{i,j,k}
      \delta_{t}\left( \frac{1}{2\det\nabla\boldsymbol{t}_t}
      \left[(\nabla\boldsymbol{t}_t)^{-T}\right]_{ij}\left[(\nabla\boldsymbol{t}_t)^{-T}\right]_{ik}v_j v_k
      \right)u_0^2.
  \end{align*}
  Since $u_0\in H^1$, all three terms are a product of an $L^1$ function which
  is independent of $t$ and a function which converges uniformly in $L^\infty$
  as $t\downarrow 0$.
  Therefore we can apply dominated convergence in~\eqref{eq:variation-dirichlet-energy-difference-quotient}
  when passing to the limit $t\downarrow 0$ and obtain
  \begin{align*}
    &\lim_{t\downarrow 0}\frac{1}{t}\left(
      \int_{\mathbb{R}^d}\frac{1}{2}|\nabla u_t|^2\,dx
      -\int_{\mathbb{R}^d}\frac{1}{2}|\nabla u_0|^2\,dx \right) \\
    =\,\,&\int_{\mathbb{R}^d}\frac{1}{2}\frac{\partial}{\partial t}\bigg|_{t=0}
           \left\lbrace
           \frac{1}{\det\nabla\boldsymbol{t}_t}
           \left| \left( \nabla\boldsymbol{t}_t \right)^{-T}
           \left( \partial_{x_i}u_0
           -u_0\tr((\nabla\boldsymbol{t}_t)^{-1}\partial_{x_i}\nabla\boldsymbol{t}_t)
           \right)_{i=1}^d
           \right|^2
           \right\rbrace\,dx \\
    =\,\,&\int_{\mathbb{R}^d}\frac{1}{2}\bigg\lbrace
           -(\nabla\cdot\xi)|\nabla u_0|^2
           +2\nabla u_0\cdot\left( -(\nabla\xi)^T\nabla u_0
           -u_0\left( \tr(\partial_t\big|_{t=0}\partial_{x_i}\nabla\boldsymbol{t}_t)
           \right)_{i=1}^d
           \right)
           \bigg\rbrace\,dx \\
    =\,\,&\int_{\mathbb{R}^d}\frac{1}{2}\bigg\lbrace
           -(\nabla\cdot\xi)|\nabla u_0|^2
           +2\nabla u_0\cdot\left( -(\nabla\xi)^T\nabla u_0
           +u_0\nabla(\nabla\cdot\xi)
           \right)
           \bigg\rbrace\,dx. \qedhere
  \end{align*}
\end{proof}

The previous lemma allows us now to compute the Euler-Lagrange equation
for the minimizing movements scheme~\eqref{eq:min-problem-uh}.

Let us first introduce some notation.
For $n\in\mathbb{N}$ let $\gamma_{n}$ be the optimal measure in the
definition~\eqref{eq:kantorovich-formulation}
for $d(u_{n},u_{n-1})$ such that $(\pi_x)_\sharp\gamma_n=u_{n-1}\mathcal{L}^d$ and
$(\pi_y)_\sharp\gamma_n=u_{n}\mathcal{L}^d$.
Further we define the flux
\begin{align*}
  j_{n}(y)&\coloneqq\int_{\mathbb{R}^d}\frac{1}{h}(x-y)\,\gamma_{n}(dx,y), \\
  j_{h}(y,t)&\coloneqq j_n(y),\qquad t\in[nh,(n+1)h),
\end{align*}
and the energy stress tensor
\begin{equation}
  \label{eq:def-energy-stress-tensor}
  \begin{split}
    \boldsymbol{T}_n
    &\coloneqq\left(\frac{1}{2}|\nabla u_n|^2+W(u_n)\right)I_d-\nabla u_n\otimes\nabla u_n, \\
    \boldsymbol{T}_h(\cdot,t)&\coloneqq\boldsymbol{T}_n,
    \qquad t\in[nh,(n+1)h).
  \end{split}
\end{equation}

\begin{lemma}
  \label{lem:EL-eq}
  The Euler-Lagrange equation for~\eqref{eq:min-problem-uh} is given by
  \begin{equation}
    \label{eq:EL-eq}
    \begin{split}
      &\int_{\mathbb{R}^d} j_{n}\cdot\xi\, dx \\
      =\,\,&\int_{\mathbb{R}^d}\boldsymbol{T}_n:\nabla\xi\,dx \\
      &-\int_{\mathbb{R}^d}
      \left[
        (\nabla\cdot\xi)\left(|\nabla u_n|^2+W'(u_n)u_n\right)
        +u_n\nabla u_n\cdot\nabla(\nabla\cdot\xi)
      \right]\,dx
    \end{split}
  \end{equation}
  for all $\xi\in C_c^\infty(\mathbb{R}^d;\mathbb{R}^d)$.
\end{lemma}

Before proving this lemma,
we state and prove two simple properties of the flux $j_n$.

\begin{lemma}\label{lem:jn-is-transport-plan}
  If $\boldsymbol{t}$ is the optimal transport plan from
  $u_{n}$ to $u_{n-1}$ as in
  Lemma~\ref{prop:existence-of-optimal-transport-maps},
  i.e., $\boldsymbol{t}_\sharp u_{n}=u_{n-1}$, then
  \begin{equation}
    \label{eq:jn-is-transport-plan}
    j_n(y)=\frac{1}{h}(\boldsymbol{t}(y)-y)u_n(y)
  \end{equation}
  for a.e.\ $y\in\mathbb{R}^d$.
\end{lemma}

\begin{proof}
  Let $\xi\in C_c^\infty(\mathbb{R}^d;\mathbb{R}^d)$ and let $\tilde u_t$ be the unique solution to the continuity equation
  \[\begin{cases}
      \partial_t\tilde u+\nabla\cdot(\tilde u\xi)=0,\quad&t>0, \\
      \tilde u(\cdot,t)|_{t=0}=u_n.
    \end{cases}
  \]
  We use the two equivalent formulations for the first variation of the Wasserstein distance.
  On the one hand, by~\cite[Thm.\ 8.4.7]{Ambrosio2008}, we have
  \begin{align*}
    \frac{d}{dt}\bigg|_{t=0}\frac{1}{2h}d^{2}(\tilde u(\cdot,t),u_{n-1})
    &=\frac{1}{h}\int_{\mathbb{R}^d\times\mathbb{R}^d}(y-x)\cdot\xi(y)\,d\gamma_n(x,y) \\
    &=-\int_{\mathbb{R}^d}j_n\cdot\xi\,dy.
  \end{align*}
  On the other hand, by~\cite[Thm.\ 8.13]{Villani2003}, we have
  \[\frac{d}{dt}\bigg|_{t=0}\frac{1}{2h}d^{2}(\tilde u(\cdot,t),u_{n-1})
    =\frac{1}{h}\int_{\mathbb{R}^d}(y-\boldsymbol{t}(y))\cdot\xi(y) u_n(y)\,dy.
  \]
  Hence
  \begin{equation}
    \label{eq:jn-is-transport-plan-distributional}
    \int_{\mathbb{R}^d}j_n\cdot\xi\,dy
    =\frac{1}{h}\int_{\mathbb{R}^d}(\boldsymbol{t}(y)-y)\cdot\xi(y) u_n(y)\,dy
  \end{equation}
  for all $\xi\in C_c^\infty(\mathbb{R}^d;\mathbb{R}^d)$.
\end{proof}

\begin{corollary}
  \label{cor:j-n-dissipation-ineq}
  For any $n$ and a.e.\ $y\in\mathbb{R}^d$ such that $u_n(y)\neq 0$ we have
  \begin{equation}
    \label{eq:j-n-dissipation-ineq}
    \frac{|j_n(y)|^2}{u_n(y)}
    \leq\int_{\mathbb{R}^d}\frac{1}{h^2}|x-y|^2\,\gamma_n(dx,y).
  \end{equation}
  Further $j_n(y)=0$ whenever $u_n(y)=0$ a.e.,\ and $\|j_n\|_{L^1}\leq \frac{1}{h}E(u_0)$.
\end{corollary}

\begin{proof}
  The inequality~\eqref{eq:j-n-dissipation-ineq} follows directly from
  Lemma~\ref{lem:jn-is-transport-plan}.
  To obtain the $L^1$ bound on $j_n$, observe that by Hölder's inequality
  and~\eqref{eq:10}
  \begin{align*}
    \int_{\mathbb{R}^d}|j_n|\,dx
    &\leq\left( \int_{\mathbb{R}^d}\frac{|j_n|^2}{u_n}\,dx \right)^{1/2}
      \left( \int_{\mathbb{R}^d}u_n\,dx \right)^{1/2}
      \leq\frac{1}{h^2}d^2(u_n,u_{n-1})
      \leq \frac{1}{h}E(u_0).\qedhere
  \end{align*}
\end{proof}

\begin{proof}[Proof of Lemma~\ref{lem:EL-eq}]
  Let $\xi\in C_c^\infty(\mathbb{R}^d;\mathbb{R}^d)$ and let
  $\{\boldsymbol{t}_t\}_{t\in\mathbb{R}}$ be the corresponding flux
  \begin{align*}
    \partial_t\boldsymbol{t}_t=\xi\circ\boldsymbol{t}_t\quad
    \text{for all  $t\in\mathbb{R}$ and $\boldsymbol{t}_0=\boldsymbol{i}_d$.}
  \end{align*}
  Define $\tilde u(\cdot,t)=(\boldsymbol{t}_t)_\sharp u_n$ as the push-forward of $u_n$ under
  $\boldsymbol{t}_t$, i.e.,
  \[u_n=\det(\nabla\boldsymbol{t}_t)\tilde
    u(\cdot,t)\circ\boldsymbol{t}_t.
  \]
  By~\cite[Thm.\ 5.34]{Villani2003}, $\tilde u$ is the unique solution to the
  continuity equation
  \begin{equation}
    \label{eq:transport-eq-un}
    \begin{cases}
      \partial_t\tilde u+\nabla\cdot(\tilde u\xi)=0,\quad&t>0, \\
      \tilde u(\cdot,t)|_{t=0}=u_n.
    \end{cases}
  \end{equation}
  To compute the first variation of the energy for $\tilde u$, we have
  \begin{align*}
    \frac{d}{dt}\bigg|_{t=0}E(\tilde u(\cdot,t))
    &=\frac{d}{dt}\bigg|_{t=0}\int_{\mathbb{R}^d}\frac{1}{2}|\nabla\tilde u(\cdot,t)|^{2}\,dx
      +\int_{\mathbb{R}^d}W'(u_n)\partial_t\tilde u(\cdot,t)|_{t=0}\,dx.
  \end{align*}
  The first term $\int_{\mathbb{R}^d}\frac{1}{2}|\nabla\tilde u|^2\,dx$ is exactly
  the term we computed in Lemma~\ref{lem:variation-dirichlet-energy}.
  For the second term
  $\int_{\mathbb{R}^d}W'(u_n)\partial_t\tilde u(\cdot,t)|_{t=0}\,dx$
  we plug in
  $\partial_t\tilde u\big|_{t=0}+\nabla\cdot(u_n\xi)=0$ and compute
  \begin{equation}
    \label{eq:second-term-energy-derivative-fix-eps}
    \begin{split}
      \int_{\mathbb{R}^d}W'(u_n)\partial_t\tilde u(\cdot,t)|_{t=0}\,dx
      &=\int_{\mathbb{R}^d}
      -W'(u_{n})\nabla\cdot(u_{n}\xi)\,dx \\
      &=\int_{\mathbb{R}^d}
      -W'(u_{n})\nabla u_{n}\cdot\xi-W'(u_{n})u_{n}\nabla\cdot\xi\,dx \\
      &=\int_{\mathbb{R}^d}
      -\nabla W(u_{n})\cdot\xi-W'(u_{n})u_{n}\nabla\cdot\xi\,dx \\
      &=\int_{\mathbb{R}^d}
      (W(u_{n})-W'(u_{n})u_{n})\nabla\cdot\xi\,dx.
    \end{split}
  \end{equation}
  Furthermore, by~\cite[Thm.\ 8.4.7]{Ambrosio2008}, we have
  \begin{align*}
    \frac{d}{dt}\bigg|_{t=0}\frac{1}{2h}d^{2}(\tilde u(\cdot,t),u_{n-1})
    &=\frac{1}{h}\int_{\mathbb{R}^d\times\mathbb{R}^d}(y-x)\cdot\xi(y)\,d\gamma_n(x,y) \\
    &=-\int_{\mathbb{R}^d}j_n\cdot\xi\,dy.
  \end{align*}
  The Euler-Lagrange equation tells us that
  \begin{equation}
    \label{eq:EL-eq-def}
    \frac{d}{dt}\bigg|_{t=0}\left(E(\tilde u(\cdot,t))+\frac{1}{2h}d^{2}(\tilde u(\cdot,t),u_{n-1})\right)=0.
  \end{equation}
  Thus combining Lemma~\ref{lem:variation-dirichlet-energy} with~\eqref{eq:second-term-energy-derivative-fix-eps}
  and~\eqref{eq:EL-eq-def} gives~\eqref{eq:EL-eq}.
\end{proof}

To improve the suboptimal a-priori estimate~\eqref{eq:discrete-energy-bounded}
to our desired sharp energy-dissipation inequality~\eqref{eq:optimal-energy-dissipation},
we need the following statement.

\begin{lemma}[{\cite[Lem.\ 4.3]{Lisini2012}}]\label{lem:heat-flow-H2-regularity}
  Let $\tilde v\in H^1(\mathbb{R}^d)$ such that $E(\tilde v)<\infty$.
  Suppose $v_t:[0,\infty)\rightarrow H^1(\mathbb{R}^d)$ is a solution to the heat flow
  \begin{equation}\label{eq:heat-eq}
    \begin{cases}
      \begin{aligned}
        \partial_t v_t&=\Delta v_t,\quad\text{on $(0,\infty)\times\mathbb{R}^d$,} \\
        v_0&=\tilde v,\quad\text{on $\mathbb{R}^d$,}
      \end{aligned}
    \end{cases}
  \end{equation}
  and
  \begin{equation}
    \label{eq:energy-derivative-heat-flow-bounded-from-below}
    \liminf_{t\downarrow 0}\frac{1}{t}\left(E(v_t)-E(v_0)\right)>-\infty.
  \end{equation}
  Then $\tilde v\in H^2(\mathbb{R}^d)$, and
  \begin{equation}
    \label{eq:neg-energy-derivative-heat-flow-bounded-from-above}
    -\liminf_{t\downarrow 0}\frac{1}{t}(E(v_t)-E(v_0))
    \geq\int_{\mathbb{R}^d}(\Delta v_0)^2\,dx-CE(v_0),
  \end{equation}
  where the constant $C$ depends only on $W$.
\end{lemma}

\begin{proof}
  By parabolic smoothing we have $v\in C^\infty(\mathbb{R}^d\times(0,\infty))$.
  Thus, for $t>0$, we can compute the derivative $\frac{d}{dt}E(v_t)$
  \begin{align*}
    \frac{d}{dt}E(v_t)
    &=\int_{\mathbb{R}^d}\nabla v_t\cdot\nabla\partial_t v_t+W'(v_t)\partial_t v_t\,dx \\
    &=-\int_{\mathbb{R}^d}(\Delta v_t)^2+W''(v_t)|\nabla v_t|^2\,dx \\
    &\leq-\int_{\mathbb{R}^d}(\Delta v_t)^2\,dx+C\int_{\mathbb{R}^d}|\nabla v_t|^2\,dx,
  \end{align*}
  where $C$ is a constant with $-C\leq W''$.
  Observing that
  \[\int_{\mathbb{R}^d}|\nabla v_t|^2\,dx
    \leq\int_{\mathbb{R}^d}|\nabla v_0|^2\,dx,
  \]
  and using that the map $t\mapsto v_t$ is continuous in $H^1(\mathbb{R}^d)$,
  we get that $t\mapsto E(v_t)$ is continuous at $t=0$,
  and by~\eqref{eq:energy-derivative-heat-flow-bounded-from-below}
  there exists a constant $C<\infty$ such that for all $t\leq t_0$ sufficiently small
  \[-C
    <\frac{1}{t}\left(E(v_t)-E(v_0)\right)
    \leq-\frac{1}{t}\int_0^t\int_{\mathbb{R}^d}(\Delta v_{s})^2\,dx\,ds+C\|\nabla v_0\|_{L^2(\mathbb{R}^d)}^2.
  \]
  Thus the family $\{\Delta v_t\}_{t\leq t_0}$ is weakly precompact
  in $L^2(\mathbb{R}^d)$.
  Since $v_t\rightarrow\tilde v$ strongly in $H^1(\mathbb{R}^d)$, we get that
  $\tilde v\in H^2(\mathbb{R}^d)$.
\end{proof}

We want to apply Lemma~\ref{lem:heat-flow-H2-regularity} with initial data
$u_n$, the solution to the minimizing movement scheme~\eqref{eq:mm-scheme}.
To verify~\eqref{eq:energy-derivative-heat-flow-bounded-from-below},
we need the flow exchange lemma below.

\begin{definition}
  Let $\mathcal{F}:\mathcal{A}\rightarrow(-\infty,\infty]$ be a proper, lower
  semi-continuous functional and let $\lambda>0$.
  Let $\Dom\mathcal{F}\coloneqq\{u\in\mathcal{A}:\mathcal{F}(u)<\infty\}$ denote
  the domain of $\mathcal{F}$.
  A continuous semi-group
  $S_t:\Dom\mathcal{F}\rightarrow\Dom\mathcal{F},\,t\geq 0$, is a
  \emph{$\lambda$-flow} for $\mathcal{F}$, if it satisfies the following
  \emph{Evolution Variational Inequality (EVI)}
  \begin{equation}
    \label{eq:evi}
    \frac{1}{2}\limsup_{t\downarrow 0}\left[
      \frac{d^2(S_t(u),v)-d^2(u,v)}{t}
    \right]
    +\frac{\lambda}{2}d^2(u,v)+\mathcal{F}(u)\leq\mathcal{F}(v)
  \end{equation}
  for all densities $u,v\in\Dom\mathcal{F}$ with $d(u,v)<\infty$.
\end{definition}

\begin{lemma}[Flow exchange lemma~{\cite[Thm.\ 3.2]{Matthes2009}}]
  \label{lem:flow-exchange-lemma}
  Let $S_t$ be a $\lambda$-flow for a proper,
  lower semi-continuous functional $\mathcal{F}$ in $\mathcal{A}$ and let
  $(u_n)_{n\geq 0}$ be a solution to the
  minimizing movements scheme~\eqref{eq:mm-scheme} with time-step size $h>0$.
  If $u_n\in\Dom\mathcal{F}$, then
  \begin{equation}
    \label{eq:flow-interchange-ineq}
    \mathcal{F}(u_n)-\mathcal{F}(u_{n-1})
    \leq h\left( \liminf_{t\downarrow 0}\frac{E(S_t(u_n))-E(u_{n})}{t} \right)
    -\frac{\lambda}{2}d^2(u_n,u_{n-1}).
  \end{equation}
\end{lemma}

We want to apply the flow exchange Lemma~\ref{lem:flow-exchange-lemma} to the
entropy functional
\begin{equation}
  \label{eq:def-entropy-functional}
  \mathcal{U}(u)\coloneqq\int_{\mathbb{R}^d}u\log u\,dx.
\end{equation}
By~\cite[Prop.\ 9.3.9]{Ambrosio2008} $\mathcal{U}$ is geodesically convex in $\mathcal{A}$
in the sense that the map $t\mapsto\mathcal{U}(u_t)$ is convex
for every geodesic $u_t$ in $\mathcal{A}$.

\begin{lemma}
  The semi-group $S_t$ induced by solutions to the heat
  equation~\eqref{eq:heat-eq} on $\mathcal{A}\cap C^\infty(\mathbb{R}^d)$
  extends to a $0$-flow for $\mathcal{U}$.
\end{lemma}

\begin{proof}
  We recall the well known fact that the heat equation~\eqref{eq:heat-eq} is the $W_2$
  gradient flow of the entropy functional
  $\mathcal{U}$, for a reference see e.g.~\cite[Thm.\ 5.1]{Jordan1998}.
  Further, since $\mathcal{U}$ is geodesically convex in $\mathcal{A}$ w.r.t.\
  the Wasserstein distance $d$,
  $\mathcal{U}$ satisfies the Evolution Variational Inequality~\eqref{eq:evi}
  with $\lambda=0$, for a reference see for example~\cite[Thm.\ 11.1.4]{Ambrosio2008}.
\end{proof}

\begin{proposition}[{\cite[Prop.\ 4.1]{Lisini2012}}]
  \label{prop:un-in-H2-and-strong-H1-convergence}
  Let $(u_n)_{n\in\mathbb{N}}$ be a solution to the minimizing movement
  scheme~\eqref{eq:mm-scheme} with time-step size $h>0$.
  Then
  \begin{enumerate}[(i)]
  \item $u_n\in H^2(\mathbb{R}^d)$ for all $n$ with a uniform bound,
  \item $u_h\rightarrow u$ strongly in $L^2(0,T;H^1(\mathbb{R}^d))$ as
    $h\downarrow 0$ for all $T>0$,
  \item $u_h\rightarrow u$ weakly in $L^2(0,T;H^2(\mathbb{R}^d))$ as
    $h\downarrow 0$ for all $T>0$.
  \end{enumerate}
\end{proposition}

To prove Proposition~\ref{prop:un-in-H2-and-strong-H1-convergence}, we need the
following lemma.

\begin{lemma}
  \label{lem:entropy-uniformly-bounded}
  There exists $\alpha<1$ and a constant $C<\infty$ depending only on $d$
  such that for all $u\in\mathcal{A}$ we have
  \begin{equation}
    \label{eq:u-log-u-bounded-from-below}
    \int_{\mathbb{R}^d}u\log u\,dx\geq-C(M_2(u)+1)^\alpha.
  \end{equation}
  Furthermore, if $(u_n)_{n\in\mathbb{N}}$ is a solution to the minimizing
  movement scheme~\eqref{eq:mm-scheme}, then there exists a constant $C<\infty$
  such that for all $n$ we have
  \begin{equation}
    \label{eq:u-log-u-mm-scheme-uniformly-bounded}
    \left| \mathcal{U}(u_n) \right|<C.
  \end{equation}
\end{lemma}

\begin{proof}
  Equation~\eqref{eq:u-log-u-bounded-from-below}
  is shown in the proof of~\cite[Prop.\ 4.1]{Jordan1998}.
  Now let $(u_n)_n$ be a solution to~\eqref{eq:mm-scheme}.
  Then, by Remark~\ref{rem:uh-uniformly-bounded-in-H1}, there exists
  $C<\infty$ such that for all $n$
  \[M_2(u_n)\leq C,\quad\int_{\mathbb{R}^d}u_n^2\,dx\leq C.\]
  Noting that
  \[\mathcal{U}(u)\leq\int_{\mathbb{R}^d}u^2\,dx\]
  for any $u\in\Dom\mathcal{U}$,
  this concludes the proof of~\eqref{eq:u-log-u-mm-scheme-uniformly-bounded}.
\end{proof}

\begin{proof}[Proof of Proposition~\ref{prop:un-in-H2-and-strong-H1-convergence}]
  We apply the flow exchange Lemma~\ref{lem:flow-exchange-lemma} with
  $\mathcal{F}=\mathcal{U}$. 
  Then \eqref{eq:energy-derivative-heat-flow-bounded-from-below} holds, and by
  Lemma~\ref{lem:heat-flow-H2-regularity} applied to $\tilde v=u_n$ we have
  $u_n\in H^2(\mathbb{R}^d)$.
  Then, by~\eqref{eq:neg-energy-derivative-heat-flow-bounded-from-above}
  and~\eqref{eq:flow-interchange-ineq}, we get
  \begin{equation}
    \label{eq:laplacian-bound}
    \frac{h}{2}\int_{\mathbb{R}^d}(\Delta u_n)^2\,dx
    \leq\mathcal{U}(u_{n-1})-\mathcal{U}(u_n)+ChE(u_n).
  \end{equation}
  Now observe that $\mathcal{U}(u_0)\leq\|u_0\|_{L^2}^2<\infty$ and recall that
  $E(u_n)\leq E(u_0)$.
  Let $N\coloneqq\lceil T/h\rceil$ and sum over $n$
  \[\|\Delta u_h\|_{L^2(0,T;L^2(\mathbb{R}^d))}^2
    \leq h\sum_{n=1}^N\int_{\mathbb{R}^d}(\Delta u_n)^2\,dx
    \leq\mathcal{U}(u_0)-\mathcal{U}(u_N)+CTE(u_0).
  \]
  By Lemma~\ref{lem:entropy-uniformly-bounded}, the RHS is uniformly bounded as
  $N\rightarrow\infty$.
  By the uniform bound on $u_h$ in $L^2(0,T;H^1(\mathbb{R}^d))$,
  see Remark~\ref{rem:uh-uniformly-bounded-in-H1},
  we get that $u_h$ is uniformly bounded in $L^2(0,T;H^2(\mathbb{R}^d))$ as
  $h\downarrow 0$ for all $T<\infty$:
  \begin{equation}
    \label{eq:uh-uniform-H2-bound}
    \int_0^T\|u_h(\cdot,t)\|_{H^2(\mathbb{R}^d)}^2\,dt\leq CTE(u_0)<\infty.
  \end{equation}
  Therefore $u_h$ converges, up to a subsequence, weakly in
  $L^2(0,T;H^2(\mathbb{R}^d))$.
  We recall that there exists $u$ such that $u_h\rightarrow u$ strongly in
  $L^2(0,T;L^2(\mathbb{R}^d))$ and weakly in $L^2(0,T;H^1(\mathbb{R}^d))$.
  By interpolation with the uniform bound~\eqref{eq:uh-uniform-H2-bound}
  we get $u_h\rightarrow u$ strongly in $L^2(0,T;H^1(\mathbb{R}^d))$.
\end{proof}

\begin{corollary}[Energy convergence]
  \label{cor:energy-convergence}
  As $l\uparrow\infty$, the energy converges:
  \[\lim_{l\uparrow\infty}\int_{0}^{T}E(u_{h_{l}}(\cdot,t))\,dt=\int_{0}^{T}E(u(\cdot,t))\,dt.\]
\end{corollary}

The following theorem establishes existence of weak solutions to the
Cahn--Hilliard equation and concludes the proof of
Theorem~\ref{thm:existence-sol-to-cahn-hilliard-eq-fix-eps}.

\begin{theorem}
  \label{thm:global-time-existence-of-weak-solutions}
  For any $T<\infty$, and $u_{h}$ given by~\eqref{eq:def-uh}
  there exists a function $u\in L^1(0,T;L^1(\mathbb{R}^d))$
  and a subsequence $u_{h_{l}}$ such that $u_{h_{l}}\rightarrow u$ in
  $L^{1}(0,T;L^1(\mathbb{R}^{d}))$ as $h_l\rightarrow 0$ and a vector field
  $j$ such that $(u,j)$ is a weak solution to the
  Cahn--Hilliard equation~\eqref{eq:cahn-hilliard-eq-fix-eps} in the sense of
  Definition~\ref{def:weak-sol-cahn-hilliard}.
\end{theorem}

\setcounter{step}{0}
\begin{proof}
  We proceed in two steps.
  First we show~\eqref{eq:weak-sol-cahn-hilliard-eq-1} and~\eqref{eq:weak-sol-cahn-hilliard-eq-2}.
  In the second step we prove the energy dissipation inequality~\eqref{eq:optimal-energy-dissipation}.
  \begin{step}
    Let $\zeta\in C_{c}^{2}(\mathbb{R}^{d}\times(0,T))$ and let
    $\zeta_{n}(x)\coloneqq\zeta(x,h_l n)$.
    Observe that, using an index shift, for sufficiently small $h_l>0$ we have
    \begin{equation}
      \label{eq:discrete-time-sum-index-shift}
      h_l\sum_{n=1}^{N}\int_{\mathbb{R}^{d}}
      \frac{u_{n}-u_{n-1}}{h_l}\zeta_{n}\,dx
      =h_l\sum_{n=0}^{N-1}\int_{\mathbb{R}^{d}}u_{n}\frac{\zeta_{n+1}-\zeta_{n}}{h_l}\,dx,
    \end{equation}
    where $N=N(h_l)=\lceil T/h_l \rceil$.
    On the one hand, in the limit $l\uparrow\infty$ the RHS converges to
    \[\int_{0}^{T}\int_{\mathbb{R}^{d}}u(x,t)\partial_{t}\zeta(x,t)\,dx\,dt,\]
    because $u_{h_l}$ is uniformly bounded in $L^1$, and the derivative
    $\partial_t\zeta$ exists.
    On the other hand, by definition of $\gamma_n$,
    \begin{align*}
      &\left|
        \int_{\mathbb{R}^{d}}\frac{u_{n}(x)-u_{n-1}(x)}{h_l}\zeta_{n}(x)dx
        -\int_{\mathbb{R}^{d}\times\mathbb{R}^{d}}\frac{y-x}{h_l}\cdot\nabla\zeta_{n}(y)d\gamma_{n}(x,y)
        \right| \\
      =\,\,&\left|\int_{\mathbb{R}^{d}\times\mathbb{R}^{d}}\frac{1}{h_l}(\zeta_{n}(y)-\zeta_{n}(x)+(x-y)
             \cdot\nabla\zeta_{n}(y))\,d\gamma_{n}(x,y)\right| \\
      \leq\,\,&\left(\sup_{\mathbb{R}^{d}}|\nabla^{2}\zeta_{n}|\right)\frac{1}{2h_l}
                \int_{\mathbb{R}^{d}\times\mathbb{R}^{d}}|x-y|^{2}d\gamma_{n}(x,y) \\
      =\,\,&\|\nabla^{2}\zeta_{n}\|_\infty\frac{1}{2h_l}d^{2}(u_{n},u_{n-1}).
    \end{align*}
    Then sum over $n$ and insert the definition of $j_n$ to obtain
    \begin{align*}
      &\left|h_l\sum_{n=1}^{N}
        \int_{\mathbb{R}^{d}}\frac{u_{n}(x)-u_{n-1}(x)}{h_l}\zeta_{n}(x)dx
        +\int_{\mathbb{R}^{d}}j_n\cdot\nabla\zeta_{n}\,dx
        \right| \\
      \leq\,\,&\|\nabla^{2}\zeta\|_{\infty}\frac{h_l}{2}\sum_{n=1}^{N}\left(\frac{d(u_{n},u_{n-1})}{h_l}\right)
                ^{2} \\
      \leq\,\,&\|\nabla^{2}\zeta\|_{\infty}h_l E(u_0)\rightarrow 0.
    \end{align*}
    Since $j_{h_l}$ is uniformly bounded in $L^1$, by~\cite[Thm.\ 1.41]{Evans2015},
    there exists a Radon measure $j$ such that
    \begin{equation}
      \label{eq:j-h-convergence}
      j_{h_l}\xrightharpoonup{*}j,\quad\text{weakly-* as Radon measures}.
    \end{equation}
    Then the LHS of~\eqref{eq:discrete-time-sum-index-shift} converges to
    \[\int_0^T\int_{\mathbb{R}^d}-j\cdot\nabla\zeta\,dx\,dt.\]
    This gives us~\eqref{eq:weak-sol-cahn-hilliard-eq-1}; of course with $\eps=1$.

    Finally,~\eqref{eq:weak-sol-cahn-hilliard-eq-2} follows immediately from
    passing to the limit $h\downarrow 0$ in the Euler-Lagrange
    equation~\eqref{eq:EL-eq}, since $j_h\xrightharpoonup{*}j$ and
    $u_h\rightarrow u$ in $L^2$ and $\nabla u_h\rightarrow\nabla u$ in $L^2$.
  \end{step}
  \begin{step}[Energy dissipation]
    The optimal energy dissipation follows from
    Proposition~\ref{prop:optimal-energy-dissipation} below. \qedhere
  \end{step}
\end{proof}

\subsection{Optimal energy dissipation}

\begin{proposition}[Optimal energy dissipation]
  \label{prop:optimal-energy-dissipation}
  The energy dissipation is optimal, i.e.,
  \begin{equation}
    \label{eq:optimal-energy-dissipation-prop}
    E(u(\cdot,T))
    +\int_{0}^{T}\int_{\mathbb{R}^d}\frac{|j|^2}{u}\,dy\,dt
    \leq E(u_{0})
  \end{equation}
  for a.e.\ $T<\infty$.
\end{proposition}

Before we prove Proposition~\ref{prop:optimal-energy-dissipation}, we introduce a few tools.

Fix $u\in\mathcal{A}$ such that $E(u)<\infty$.
For $\tau>0$ let
\[e_\tau(u)\coloneqq\inf_{v\in\mathcal{A}}\{E(v)+\frac{1}{2\tau}d^2(u,v)\},\quad
  \mathcal{J}_{\tau}(u)\coloneqq\argmin_{v\in\mathcal{A}}\{E(v)+\frac{1}{2\tau}d^2(u,v)\}.
\]
By Lemma~\ref{lem:existence-of-minimizers-mm-fix-eps}, $\mathcal{J}_\tau(u)$ is non-empty for all $\tau>0$,
and we define
\[z^+_{\tau}(u)\coloneqq\sup_{u_{\tau}\in\mathcal{J}_{\tau}(u)}d(u,u_\tau),\quad
  z^-_{\tau}(u)\coloneqq\inf_{u_{\tau}\in\mathcal{J}_{\tau}(u)}d(u,u_\tau).
\]
Then $z^+_\tau(u)=z^-_\tau(u)$ for almost all $\tau$ and, by~\cite[Thm.\ 3.1.4]{Ambrosio2008},
\begin{equation}\label{eq:derivative-metric-interpolation-mm-scheme}
  \frac{d}{d\tau}e_\tau(u)
  =-\frac{(z^{\pm}_\tau(u))^2}{2\tau^2}.
\end{equation}
In particular
\begin{equation}\label{eq:metric-interpolation-integral}
  \frac{d^2(u_\tau,u)}{2\tau}+\int_0^\tau\frac{(z^{\pm}_t(u))^2}{2t^2}\,dt
  =E(u)-E(u_\tau)\quad\text{for all $u_\tau\in\mathcal{J}_\tau(u)$.}
\end{equation}

\begin{proof}[Proof of Proposition~\ref{prop:optimal-energy-dissipation}]
  For $h>0$ and $(u_n)_n$ a solution to the minimizing movements scheme and $j_n$ as before,
  with piecewise constant time interpolation $u_h$ and $j_h$,
  we define a new interpolation as follows:
  Define $\tilde u_h(t)$ by
  \begin{equation}
    \label{eq:def-de-giorgi-interpolation}
    \tilde u_h(t)\coloneqq\tilde{u}_{n,\tau},\quad t=nh+\tau,
  \end{equation}
  where $0<\tau<h$ and $\tilde u_{n,\tau}\in\mathcal{J}_\tau(u_n)$, i.e.,
  $\tilde u_{n,\tau}$ minimizes $v\mapsto E(v)+\frac{1}{2\tau}d^2(u_n,v)$.
  \begin{figure}[h]
    \centering
    \begin{tikzpicture}[scale=0.8]
      \draw[black,->] (0,0) -- (5,0) node[anchor=west] {$t$};
      \draw[black,->] (0,0) -- (0,4) node[anchor=south] {$\mathcal{A}$};
      \foreach \n in {1,2,3,4}
      \draw (\n cm, 1pt) -- (\n cm,-1pt) node[anchor=north] {$\n h$};

      \draw[thick,black] (0,2) node[black,anchor=east] {$u_{0}$} -- (1,2) node[black,anchor=north,pos=.5] {$u_h$};
      \foreach \x/\y in {1/2.5, 2/1.5, 3/1, 4/1.6} {
        \draw[thick,black] (\x,\y) node[black,anchor=north] {$u_{\x}$} -- (\x+1,\y);
      }
      \draw[thick,red]
      (0,2) .. controls (0.3,2.35) .. (1,2.5) node[red,anchor=south east] {$u_{0,\tau}$}
      .. controls (1.6,2.2) .. (2,1.5) node[red,anchor=east] {$u_{1,\tau}$}
      .. controls (2.3,1.2) .. (3,1) node[red,anchor=north east] {$u_{2,\tau}$}
      .. controls (3.7,1.5) .. (4,1.6) node[red,anchor=south east] {$u_{3,\tau}$}
      .. controls (4.5,2.3) .. (5,2.4) node[red,anchor=south east] {$u_{4,\tau}$};
    \end{tikzpicture}
    \caption{Illustration of interpolation functions}
    \label{fig:interpolation-functions}
  \end{figure}
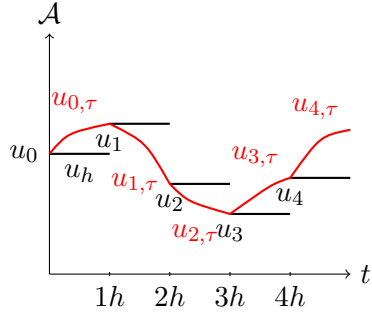

  We can obtain a uniform bound on $d(u_h(t),\tilde u_h(t))$:
  If $t=nh+\tau$ with $\tau<h$, then
  \begin{align*}
    d^2(u_h(\cdot,t),\tilde u_h(\cdot,t))
    =d^2(u_n,\tilde u_{n,\tau})
    \leq 2\tau E(u_n)
    \leq 2\tau E(u_0)
  \end{align*}
  for all $n$ and all $0<\tau<h$.
  Hence we have $d(u_h(\cdot,t),\tilde u_h(\cdot,t))\leq\sqrt{2E(u_0)}\sqrt{h}$.
  Since $u_h\rightarrow u$ in $L^1(0,T;L^1(\mathbb{R}^d))$,
  we have $u(\cdot,t)\in\mathcal{A}$ for a.e.\ $t$
  and $\int_0^T d^2(u_h(\cdot,t),u(\cdot,t))\,dt\rightarrow 0$.
  Therefore
  \begin{equation}
    \label{eq:tilde-u-h-converges-to-u}    
    \int_0^T d^2(\tilde u_{h},u)\,dt
    \leq\int_0^T d^2(u_{h},u)+d^2(\tilde u_{h},u_h)\,dt
    \rightarrow 0,
  \end{equation}
  and since $d$ metrizes weak-* convergence on $\mathcal{A}$, we get that
  \[\tilde u_h\xrightharpoonup{*}u.\]
  To prove the optimal dissipation
  inequality~\eqref{eq:optimal-energy-dissipation-prop},
  we want to use~\eqref{eq:metric-interpolation-integral}.
  More precisely, we apply~\eqref{eq:metric-interpolation-integral} to $u=u_n$
  and $\tau=h$ and sum over $n$:
  \begin{align*}
    h\sum_{n=1}^N\frac{d^2(u_n,u_{n-1})}{2h}
    +h\sum_{n=0}^N\int_0^h\frac{d^2(u_n,\tilde u_{n,\tau})}{2\tau^2}\,d\tau
    \leq E(u_0)-E(u_h(Nh)).
  \end{align*}
  As with the piecewise contant time interpolation, we define
  \begin{equation}
    \label{eq:def-j-n-de-giorgi}
    \begin{split}
      \tilde j_{n,\tau}(y)&\coloneqq\frac{1}{\tau}\int_{\mathbb{R}^d}(x-y)\tilde\gamma_{n,\tau}(dx,y), \\
      \tilde j_h(\cdot,t)&\coloneqq\tilde j_{n,\tau},\quad t=nh+\tau,
    \end{split}
  \end{equation}
  where $\tilde\gamma_{n,\tau}$ is the optimal measure in the definition of $d(\tilde u_{n,\tau},u_{n})$
  with $(\pi_x)_\sharp\tilde\gamma_{n,\tau}=u_{n}\mathcal{L}^d$.

  Now we note that by Corollary~\ref{cor:j-n-dissipation-ineq} we have the lower
  bound
  \begin{align*}
    &\,\,\frac{1}{2}\int_0^T\int_{\mathbb{R}^d}\frac{|j_h|^2}{u_h}\,dx\,dt
      +\frac{1}{2}\int_0^T\int_{\mathbb{R}^d}\frac{|\tilde j_h|^2}{\tilde u_h}\,dx\,dt \\
    \leq&\,\, h\sum_{n=1}^N\frac{d^2(u_n,u_{n-1})}{2h}
          +h\sum_{n=0}^N\int_0^h\frac{d^2(u_n,\tilde u_{n,\tau})}{2\tau^2}\,d\tau.
  \end{align*}
  Since $(u,j)\mapsto\frac{|j|^2}{u}$ is jointly convex, the functional
  $\int\frac{|j|^2}{u}$ is lower semi-continuous.
  Hence our claim~\eqref{eq:optimal-energy-dissipation-prop} follows once we have
  shown that $\tilde u_h\rightarrow u$ in $L^2(0,T;H^1(\mathbb{R}^d))$
  and $\tilde j_h\xrightharpoonup{*}j$.
  
  Indeed, by Corollary~\ref{cor:j-n-dissipation-ineq}
  and~\eqref{eq:metric-interpolation-integral},
  $\tilde j_{h}$ is uniformly bounded in
  $L^1(0,T;L^1(\mathbb{R}^d;\mathbb{R}^d))$,
  hence $\tilde j_h\xrightharpoonup{*}\tilde j$ for some $\tilde j$
  as Radon measures.
  Let $\xi\in C_c^\infty(\mathbb{R}^d\times(0,T);\mathbb{R}^d)$ and let
  $\xi_{n,\tau}\coloneqq\xi(\cdot,nh+\tau)$.
  By the Euler-Lagrange Equation~\eqref{eq:EL-eq} we have for all $n$ and all $0<\tau<h$
  \begin{align*}
    \int_{\mathbb{R}^d}\tilde j_{n,\tau}\cdot\xi_{n,\tau}\,dx
    &=\int_{\mathbb{R}^d}\tilde{\boldsymbol{T}}_{n,\tau}:\nabla\xi_{n,\tau}\,dx
      -\int_{\mathbb{R}^d}\tilde F_{n,\tau}\,dx,
  \end{align*}
  where $\tilde{\boldsymbol{T}}_{n,\tau}$ denotes the energy stress tensor
  as defined in~\eqref{eq:def-energy-stress-tensor} for $\tilde u_{n,\tau}$,
  and 
  $\tilde F_{n,\tau}=F(\tilde u_{n,\tau},\nabla\tilde u_{n,\tau},\nabla\xi,\nabla^2\xi)$
  is the second right-hand side term in the Euler-Lagrange
  equation~\eqref{eq:EL-eq}.
  Now let $T=Nh$ for some $N$ and sum over $n$:
  \begin{equation}
    \label{eq:EL-eq-tilde-j-interpolation}
    \begin{split}
      \int_0^T\int_{\mathbb{R}^d}\tilde j_h\cdot\xi\,dx\,dt
      &=h\sum_{n=0}^N\int_0^h\int_{\mathbb{R}^d}\tilde j_{n,\tau}\cdot\xi_{n,\tau}\,dx\,d\tau \\
      &=h\sum_{n=0}^N\left(
        \int_0^h\int_{\mathbb{R}^d}
        \tilde{\boldsymbol{T}}_{n,\tau}:\nabla\xi_{n,\tau}\,dx\,d\tau
        -\int_0^h\int_{\mathbb{R}^d}\tilde F_{n,\tau}\,dx\,d\tau
      \right).
    \end{split}
  \end{equation}
  Thus it suffices to show that $\tilde u_h\rightarrow u$ in $L^2(0,T;H^1(\mathbb{R}^d))$.
  Indeed, by the flow exchange Lemma~\ref{lem:flow-exchange-lemma} applied to
  entropy functional $\mathcal{F}=\mathcal{U}$,
  we can apply Lemma~\ref{lem:heat-flow-H2-regularity}
  also to $\tilde u_{n,\tau}$ (instead of $u_n$) as before
  to obtain $\tilde u_{n,\tau}\in H^2(\mathbb{R}^d)$ for all $n$
  and all $0<\tau<h$.
  By Remark~\ref{rem:uh-uniformly-bounded-in-H1}, $\tilde u_h(nh)=u_n$
  is uniformly bounded in $H^1$.
  By the estimate $E(\tilde u_{n,\tau})\leq E(u_n)\leq E(u_0)$,
  we get that there exists $C<\infty$ depending only on $W$ such that
  \[\int_0^T\|\tilde u_h\|_{H^1(\mathbb{R}^d)}\,dt\leq CTE(u_0),\]
  and as in~\eqref{eq:laplacian-bound},
  \begin{align*}
    \frac{\tau}{2}\int_{\mathbb{R}^d}(\Delta\tilde u_{n,\tau})^2
    \leq\mathcal{U}(u_n)-\mathcal{U}(\tilde u_{n,\tau})+C\tau E(\tilde u_{n,\tau}).
  \end{align*}
  Hence $\tilde u_h$ is uniformly bounded in $L^2(0,T;H^2(\mathbb{R}^d))$,
  and $\tilde u_{h}\rightarrow u$ in $L^2(0,T;H^1(\mathbb{R}^d))$.
  Therefore the right-hand side of~\eqref{eq:EL-eq-tilde-j-interpolation}
  converges to the same limit as the one for $u_h$ instead of $\tilde u_h$, and hence
  $\tilde j_h\xrightharpoonup{*} j$,
  where $j$ is the same as in
  Theorem~\ref{thm:global-time-existence-of-weak-solutions}.
  Now~\eqref{eq:optimal-energy-dissipation-prop} follows from
  \begin{align*}
    \int_0^T\int_{\mathbb{R}^d}\frac{|j|^2}{u}\,dy\,dt
    &\leq\liminf_{h\downarrow 0}\left(
      \frac{1}{2}\int_0^T\int_{\mathbb{R}^d}\frac{|j_h|^2}{u_h}\,dx\,dt
      +\frac{1}{2}\int_0^T\int_{\mathbb{R}^d}\frac{|\tilde j_h|^2}{\tilde u_h}\,dx\,dt
      \right) \\
    &\leq E(u_0)-E(u(\cdot,T)).\qedhere
  \end{align*}
\end{proof}

\section{The sharp interface limit}\label{sec:eps-limit}

In Section~\ref{sec:fix-eps-problem} we have shown that for
every $\eps>0$ with initial data $u_{\eps,0}\in\mathcal{A}$ such that $E_\eps(u_{\eps,0})<\infty$, there exists a
weak solution $u_\eps$ to the Cahn--Hilliard equation~\eqref{eq:cahn-hilliard-eq-fix-eps}
in the sense of Definition~\ref{def:weak-sol-cahn-hilliard}.
In this section we show that there exists a subsequence $\eps_l\rightarrow 0$ and a
function $u:\mathbb{R}^d\times[0,\infty)\rightarrow\mathbb{R}$ such that
$u_{\eps_l}\rightarrow u$ in $L^1$ as $l\uparrow\infty$
under the following \emph{well preparedness} condition on the initial data $u_{\eps,0}$:
\begin{equation}
  \label{eq:initial-data-well-prepared}
  \begin{split}
    u_{\eps,0}&\rightarrow\chi_{\Omega_0}\quad\text{in $L^1$,} \\
    E_\eps(u_{\eps,0})&\rightarrow\sigma P(\Omega_0), \\
    \sup_\eps M_2(u_{\eps,0})&<\infty.
  \end{split}
\end{equation}
Throughout this section we assume that the initial data are well prepared,
and we set $E_0\coloneqq\sup_{\eps}E_\eps(u_{\eps,0})$.

\begin{theorem}
  \label{thm:u-eps-convergence-to-weak-hele-shaw-solution}
  Let $(u_\eps,j_\eps)$ be weak solutions to the Cahn--Hilliard
  equation~\eqref{eq:cahn-hilliard-eq-fix-eps} in the sense of Definition~\ref{def:weak-sol-cahn-hilliard}
  with initial data $u_{\eps,0}$ satisfying~\eqref{eq:initial-data-well-prepared}.

  Then there exists a subsequence $\eps_l\rightarrow 0$ and a
  family of finite perimeter sets $(\Omega(t))_{t\in[0,T]}$ such that the
  following hold:
  \begin{enumerate}[(i)]
  \item For almost all $t\in[0,T]$ we have
    \begin{equation}
      \label{eq:u-eps-L1-convergence}
      u_{\eps_l}\xrightarrow{l\uparrow\infty}\chi_{\Omega}\qquad\text{in $L^1(0,T;L^1(\mathbb{R}^d))$,}
    \end{equation}
    where $\chi_{\Omega}(x,t)\coloneqq\chi_{\Omega(t)}(x)$.
  \item There exists $j\in L^2(0,T;L^2(\mathbb{R}^d;\mathbb{R}^d))$ such that
    \begin{equation}
      \label{eq:characteristic-function-limit}
      j_{\eps_l}\xrightharpoonup{*}\chi_{\Omega(t)} j
      \quad\text{as Radon measures.}
    \end{equation}
  \item If in addition to~\eqref{eq:u-eps-L1-convergence} and~\eqref{eq:characteristic-function-limit},
    \begin{equation}
      \label{eq:perimeter-bounded-from-below-by-energy}
      \limsup_{l\uparrow\infty}\int_0^T E_{\eps_l}(u_{\eps_l}(\cdot,t))\,dt
      \leq\int_0^T\sigma P(\Omega(t))\,dt,
    \end{equation}
    then $(\Omega,j)$ is a weak solution to the Hele--Shaw flow in the sense of
    Definition~\ref{def:weak-sol-hele-shaw} below.
  \end{enumerate}
\end{theorem}

Before we give the definition of weak solutions, we introduce some standard
notation, see also~\cite{Maggi2012}.
For an open set $\Omega\subset\mathbb{R}^d,\,BV(\Omega)$ denote the space of
functions $u\in L^1(\Omega)$ with bounded variation in $\Omega$.
If $\Omega$ is a set of finite perimeter, we denote by $\partial^*\Omega$ the
\emph{reduced boundary}.
The \emph{measure theoretic outward unit normal} is denoted by
$\nu_\Omega:\partial^*\Omega\rightarrow S^{d-1}$.
The Gauss-Green measure of $\Omega$ will be denoted by $\mu_\Omega$, and
$\mathcal{H}^{d-1}$ denotes the $d-1$-dimensional Hausdorff measure.

\begin{definition}[Weak solution of the Hele--Shaw flow]
  \label{def:weak-sol-hele-shaw}
  Let $d\geq 2$ and $T\in(0,\infty)$.
  Let $\Omega\coloneqq\{\Omega(t)\}_{t\in[0,T]}$ be a family of finite perimeter sets
  and let $j\in L^2(0,T;L^2(\mathbb{R}^d;\mathbb{R}^d))$.
  We say that the pair $(\Omega,j)$ is a
  \emph{weak solution to the Hele--Shaw flow}, if
  \begin{enumerate}[(i)]
  \item For all $\zeta\in C_c^1(\mathbb{R}^d\times[0,T))$ we have
    \begin{equation}
      \label{eq:def-weak-sol-hele-shaw-conservation-eq}
      \int_{\mathbb{R}^d}\chi_{\Omega_0}\zeta(\cdot,t)\,dx
      +\int_0^T\int_{\mathbb{R}^d}\chi_{\Omega(t)}\partial_t\zeta
      +\chi_{\Omega(t)}j\cdot\nabla\zeta\,dx\,dt=0,
    \end{equation}
    where $\chi_\Omega(x,t)=\chi_{\Omega(t)}(x)$.
  \item For all $\xi\in C_c^1(\mathbb{R}^d\times(0,T);\mathbb{R}^d)$ with $\nabla\cdot\xi=0$ we have
    \begin{equation}
      \label{eq:def-weak-sol-hele-shaw-curvature-eq}
      \int_0^T\int_{\Omega(t)}\xi\cdot j(\cdot,t)\,dx\,dt
      =\int_0^T\int_{\partial^*\Omega(t)}
      (\nabla\cdot\xi-\nu_{\Omega(t)}\cdot\nabla\xi\,\nu_{\Omega(t)})\,d\mathcal{H}^{d-1}\,dt.
    \end{equation}
  \item For a.e.\ $T'\in[0,T]$ we have
    \begin{equation}
      \label{eq:optimal-energy-dissipation-eps-limit}
      \sigma P(\Omega(T'))
      +\int_0^{T'}\int_{\Omega(t)}|j|^2\,dx\,dt
      \leq\sigma P(\Omega_0).
    \end{equation}
  \end{enumerate}
\end{definition}

This definition is motivated by the classical definition of the Hele shaw flow.
If $j$ and $\Omega$ are smooth and a weak solution to the Hele--Shaw flow,
then $(\Omega,j)$ should solve the Hele--Shaw flow in the classical sense, that is,
$j$ is a strong solution to~\eqref{eq:velocity-eq}--\eqref{eq:conservation-eq}.
This is shown in
Lemma~\ref{lem:weak-sol-is-classical-sol-if-smooth} below.

\subsection{Compactness}
\label{sec:sharp-interface-limit-compactness}

In this section we prove (i) and (ii)
of Theorem~\ref{thm:u-eps-convergence-to-weak-hele-shaw-solution}.

Before we go into the proof, we recall a few definitions and perform some preliminary computations,
which we will need to complete the proof.

Let us recall the notion of convergence in measure:
\begin{definition}
  Let $B$ be a separable Banach space and let $\mathcal{M}([0,T];B)$ denote the
  space of measurable $B$-valued functions on $[0,T]$.
  A sequence $\{u_n\}_{n\in\mathbb{N}}\subset\mathcal{M}([0,T];B)$
  \emph{converges in measure} to $u\in\mathcal{M}([0,T];B)$, if
  \begin{equation}
    \label{eq:convergence-in-measure}
    \lim_{n\uparrow\infty}\bigg| \bigg\lbrace t\in(0,T):\|u_n(t)-u(t)\|_{B}\geq\sigma\bigg\rbrace
    \bigg|=0\quad\text{for all $\sigma>0$.}
  \end{equation}
\end{definition}

To get precompactness of $u_\eps$ in $L^1$,
an important tool will a be a variant of the
Aubin-Lions Lemma, see Theorem~\ref{thm:rossi-savare-compactness-in-measure} below.
This theorem relies on the existence of a normal coercive integrand $\mathcal{F}$,
so we briefly recall the definition.

\begin{definition}
  Let $B$ be a separable Banach space and let $T>0$.
  A functional $\mathcal{F}:(0,T)\times B\rightarrow[0,\infty]$ is \emph{coercive}, if $\{u\in B:\mathcal{F}_{t}(u)
  \leq c\}$ is compact for all $c<\infty$ and a.e. $t\in(0,T)$.
  Further, $\mathcal{F}$ is a \emph{normal integrand}, if
  \begin{enumerate}
  \item $\mathcal{F}$ is $\mathcal{L}\otimes\mathcal{B}(B)$-measurable, and
  \item the maps $u\mapsto F_{t}(u)$ are lower semi-continuous for a.e. $t\in(0,T)$.
  \item If in addition, $g:B\times B\rightarrow[0,\infty]$ is a lower semi-continuous map,
    we say that \emph{$g$ is compatible with $\mathcal{F}$}, if the following holds for a.e.\ $t\in(0,T)$:
    If $u,v\in B$ such that $\mathcal{F}(t,u),\mathcal{F}(t,v)<\infty$, then
    \begin{equation}
      \label{eq:g-compatibility-condition}
      u=v\quad\text{whenever $g(u,v)=0$.}
    \end{equation}
  \end{enumerate}
\end{definition}

\begin{theorem}[{\cite[Thm.\ 2]{Rossi2003}}]
  \label{thm:rossi-savare-compactness-in-measure}
  Let $B$ be a separable Banach space and let $\mathcal{U}$ be a family of
  measurable $B$-valued functions on $(0,T)$.
  If there exists a normal coercive integrand $\mathcal{F}:(0,T)\times B\rightarrow[0,\infty]$ and a l.s.c.\ map
  $g:B\times B\rightarrow[0,\infty]$ compatible with $\mathcal{F}$ such that
  \begin{equation}
    \label{eq:tightness-condition}
    \mathcal{U}\text{ is tight w.r.t. $\mathcal{F}$, i.e., }
    S\coloneqq\sup_{u\in\mathcal{U}}\int_{0}^{T}\mathcal{F}(t,u(t))dt<\infty,
  \end{equation}
  and
  \begin{equation}
    \label{eq:7}
    \lim_{h\downarrow 0}\sup_{u\in\mathcal{U}}\int_{0}^{T-h}g(u(t+h),u(t))dt=0,
  \end{equation}
  then $\mathcal{U}$ is precompact in $\mathcal{M}([0,T];B)$.
\end{theorem}

For us, $B$ will be $L^1(\mathbb{R}^d)$,
and $g$ will be the Wasserstein distance:
Let
\[g:L^1(\mathbb{R}^d)\times L^1(\mathbb{R}^d)\rightarrow[0,\infty],\qquad
  \begin{cases}
    g(u,v)=W_{2}(u\mathcal{L}^{d},v\mathcal{L}^{d}),
    \quad&\text{if $u,v\in\mathcal{A}$,} \\
    \infty,&\text{else.}
  \end{cases}
\]
where $W_{2}:\mathcal{P}_{2}\times\mathcal{P}_{2}\rightarrow[0,\infty]$ is the
Wasserstein-distance.
Then $g$ is a metric
on $L^1(\mathbb{R}^d)\cap\mathcal{A}$,
so the compatibility condition~\eqref{eq:g-compatibility-condition}
holds.

To define $\mathcal{F}$, we let
\[\phi:\mathbb{R}\rightarrow\mathbb{R},\quad
  \phi(s)=\int_0^s\sqrt{2W(r)}\,dr,
\]
and
\[\mathcal{F}:L^{1}(\mathbb{R}^{d})\rightarrow[0,\infty],\quad
  \mathcal{F}(u)=\int_{\mathbb{R}^{d}}|\nabla(\phi\circ u)|\,dx.
\]
We start with an upper bound on $\mathcal{F}$ in terms of the energy $E_\eps$,
as well as a bound on the energy $E_\eps$, which we need later.

Let $u\in\mathcal{A}$.
By Young's inequality $2ab\leq a^{2}+b^{2}$ with
$a=\frac{1}{\sqrt{\eps}}\sqrt{2W(u)}$ and
$b=\sqrt{\eps}|\nabla u|$, we get
\begin{equation}
  \label{eq:modica-mortola-young-ineq}
  |\nabla(\phi\circ u)|
  =|\phi'(u)||\nabla u|
  \leq\frac{\eps}{2}|\nabla u|^2+\frac{1}{\eps}W(u).
\end{equation}
Then we have
\begin{equation}
  \label{eq:modica-mortola-trick}
  \mathcal{F}(u)
  =\int_{\mathbb{R}^d} |\nabla(\phi\circ u)|dx
  \leq E_{\eps}(u).
\end{equation}
It follows that
\begin{equation}
  \label{eq:3}
  \int_{\mathbb{R}^{d}}|(\phi\circ u)(x+z)-(\phi\circ u)(x)|dx
  \leq|z|E_{\eps}(u).
\end{equation}
In particular, if $u=u_{\eps}(\cdot,t)$ is a weak solution to the Cahn--Hilliard equation, we can
use~\eqref{eq:optimal-energy-dissipation},
and integrate~\eqref{eq:modica-mortola-trick} over $(0,T)$ to obtain the estimate
\begin{equation}
  \label{eq:seq-tightness}
  \sup_{\eps}\int_{0}^{T}\mathcal{F}(u_{\eps}(\cdot,t))\,dt
  \overset{\eqref{eq:modica-mortola-trick}}{\leq}\sup_{\eps}\int_{0}^{T}E_{\eps}(u_{\eps}(\cdot,t))\,dt
  \overset{\eqref{eq:optimal-energy-dissipation}}{\leq}T\sup_{\eps}E(u_{\eps,0})
  \leq TE_{0}.
\end{equation}
Hence, if $\mathcal{U}=\{u_\eps|_{[0,T]}\}_\eps$, then the tightness
condition~\eqref{eq:tightness-condition} holds for all $T>0$.
It remains to check~\eqref{eq:7} and finding a normal coercive integrand.

To check condition~\eqref{eq:7}, we use the following lemma.
\begin{lemma}
  [A H{\"o}lder bound on $g$]
  There exists a constant $C>0$ depending only on $T$ and $E_\eps(u_{\eps,0})$, such that
  for any $\eps>0$ and $0\leq t_0<t_1\leq T$, it holds
  \[g(u_{\eps}(\cdot,t_{1}),u_{\eps}(\cdot,t_{0}))\leq C|t_{1}-t_{0}|^{1/2}.\]
\end{lemma}
Then~\eqref{eq:7} follows immediately, since
\[\lim_{h\downarrow 0}\sup_{\eps}\int_{0}^{T-h}g(u_{\eps}(\cdot,t+h),u_{\eps}(\cdot,t))\,dt
  \leq\lim_{h\downarrow 0}C\int_{0}^{T-h}h^{1/2}\,dt
  =0.
\]
\begin{proof}
  By the Benamou-Brenier formula~\cite[Thm.\ 8.1]{Villani2003}, we have
  \begin{equation}
    \label{eq:benamou-benier-formula}
    g^2(u(\cdot,t_{0}),u(\cdot,t_{1}))
    =\inf_{\rho,v}\{\int_{0}^{1}\int_{\mathbb{R}^d}\rho(x,t)|v(x,t)|^{2}dx\,dt\},
  \end{equation}
  where we minimize over all $\rho,v$ such that $\partial_{t}\rho+\nabla\cdot(\rho v)=0$ and
  \begin{align*}
    \rho(\cdot,0)&=u(\cdot,t_{0}), \\
    \rho(\cdot,1)&=u(\cdot,t_{1}).
  \end{align*}
  Now choose
  \[\rho_{\eps}(x,t)\coloneqq u_{\eps}(x,t(t_{1}-t_{0})+t_{0}),\quad
    v_{\eps}(x,t)\coloneqq(t_{1}-t_{0})\frac{j_\eps(x,t(t_{1}-t_{0})+t_{0})}{u_\eps(x,t(t_{1}-t_{0})+t_{0})},
  \]
  which are admissible for the infimum, since by~\eqref{eq:cahn-hilliard-eq-fix-eps}
  \[\partial_{t}\rho_{\eps}+\nabla\cdot(\rho_{\eps}v_{\eps})=0\]
  in the sense of distributions.
  Then by~\eqref{eq:benamou-benier-formula},~\eqref{eq:optimal-energy-dissipation}, and the change of variables
  formula
  \begin{align*}
    g^{2}(u(\cdot,t_{0}),u(\cdot,t_{1}))
    &\leq
      \int_{0}^{1}\int_{\mathbb{R}^d}\rho_{\eps}(x,t)|v_{\eps}(x,t)|^{2} dx\,dt \\
    &=(t_{1}-t_{0})\int_{t_{0}}^{t_{1}}\int_{\mathbb{R}^d}\frac{|j_\eps|^2}{u_\eps}\,dx\,dt \\
    &=(t_{1}-t_{0})E_{\eps}(u_{\eps}(\cdot,0)) \\
    &\leq (t_{1}-t_{0})E_{0}.\qedhere
  \end{align*}
\end{proof}

To apply Theorem~\ref{thm:rossi-savare-compactness-in-measure},
we slightly adjust the integrand $\mathcal{F}$.
For $D\subset L^1(\mathbb{R}^d)$, define
\[\mathcal{F}_D(u)
  \coloneqq\begin{cases}
    \int_{\mathbb{R}^{d}}|\nabla(\phi\circ u)|\,dx,
    \quad&\text{if }u\in D,\,\phi\circ u\in BV(\mathbb{R}^d), \\
    \infty,&\text{else.}
  \end{cases}
\]
Then, with suitable assumptions on $D$,
$\mathcal{F}_D$ is a normal coercive integrand on $L^1(\mathbb{R}^d)$,
as the following lemma shows.

\begin{lemma}
  \label{lem:F_normal_coercive_integrand}
  Let $D\subset L^1(\mathbb{R}^d)$ be closed such that
  $\int_{\mathbb{R}^d}u\,dx=1$ and $u\geq 0$
  for all $u\in D$, and
  \[\sup_{u\in D}\|u\|_{L^4(\mathbb{R}^d)}<\infty,\quad
    \sup_{u\in D}M_2(u)<\infty.\]
  Then $\mathcal{F}_D$ is a normal coercive integrand on $L^1(\mathbb{R}^d)$.
\end{lemma}
Note that for our purposes, if suffices to choose
\[D=\{u\in L^1(\mathbb{R}^d):\int_{\mathbb{R}^d}u\,dx=1,\,\|u\|_{L^4}\leq C,\,M_2(u)\leq C',\,
  \phi\circ u\in BV(\mathbb{R}^d)\},
\]
for some $C,C'<\infty$.
Indeed, $D$ is closed, since the constraints $\int u\,dx=1,\,\|u\|_{L^4}\leq C$ and $M_2(u)\leq C'$ are continuous,
and the total variation is lower semi-continuous.

\setcounter{step}{0}
\begin{proof}
  Since $\mathcal{F}_D$ is clearly measurable, we only need to show that $\mathcal{F}_D$
  is lower semi-continuous and coercive.
  \begin{step}[coercivity]
    Let $c>0$ and let $(u_{k})_{k\geq 1}\subset D$ be a sequence with $\mathcal{F}_D(u_{k})\leq c$,
    and define
    \[w_{k}\coloneqq\phi\circ u_{k}:\mathbb{R}^{d}\rightarrow\mathbb{R}.\]
    We need to show that there exists a subsequence $u_{k_{l}}$ and
    $u:\mathbb{R}^{d}\rightarrow\mathbb{R}$ such that
    \[u_{k_{l}}\rightarrow u\quad\text{in $L^{1}(\mathbb{R}^d)$,}\]
    and $\mathcal{F}_D(u)\leq c$.
    The latter will follow from lower semi-continuity of $\mathcal{F}_D$ (see Step 2).
    First we show that there exists a subsequence $w_{k_{l}}$ such that
    $w_{k_{l}}\rightarrow w$ pointwise a.e.\ in $\mathbb{R}^{d}$ for some $w$.
    By assumption, we have
    \[\int_{\mathbb{R}^d}|\nabla w_{k}|\,dx=\mathcal{F}_D(u_{k})\leq c.\]
    In particular, we have $u_k\in D$ for all $k$.
    Further there exist $M,C<\infty$ such that
    \begin{align*}
      \int_{\mathbb{R}^d}|w_k|\,dx
      &=\int_{\mathbb{R}^d}|\phi(u_k)|\,dx \\
      &\leq\int_{\{|u_k|>M\}}\frac{1}{6}|u_k|^3+O(u^2)\,dx
        +\int_{\{|u_k|\leq M\}}|\phi(u_k)|\,dx \\
      &\leq\int_{\{|u_k|>M\}}C|u_k|^3\,dx
        +\left(\Lip\phi|_{[0,M]}\right)\int_{\mathbb{R}^d}|u_k|\,dx<\infty.
    \end{align*}
    Thus $w_k\in W^{1,1}(\mathbb{R}^d)$ for all $k$ and the sequence $w_k$
    is uniformly bounded in $W^{1,1}$, because $u_k$ is uniformly bounded in
    $L^3$ (by interpolation between $L^1$ and $L^4$).
    By Rellich's Theorem~\cite[Chap.\ 5.7, Thm.\ 1]{Evans2010} and a diagonal argument,
    there exists a subsequence converging in
    $L^{1}_{loc}(\mathbb{R}^d)$ to some $w$.
    Passing to a further subsequence if necessary,
    we may assume that $w_{k_l}\rightarrow w$ a.e.\ in $\mathbb{R}^d$.
    Now define $u\coloneqq\phi^{-1}\circ w$.
    Then
    \[u_{k_l}(x)=\phi^{-1}(w_{k_l}(x))\rightarrow\phi^{-1}(w(x))=u(x)
      \quad\text{for a.e.\ $x\in\mathbb{R}^{d}$.}
    \]
    To get $L^1$ convergence of $u_k$,
    it suffices to show that~\cite[Thm.\ 4.16]{Alt2016}
    \begin{align}
      \sup_k\int_{\mathbb{R}^d}|u_k(x-z)-u_k(x)|\,dx
      \xrightarrow{z\rightarrow 0}0,\label{eq:frechet-kolmogorov-eq1} \\
      \sup_k\int_{\mathbb{R}^d\setminus B_R(0)}|u_k|\,dx
      \xrightarrow{R\uparrow\infty}0.
      \label{eq:frechet-kolmogorov-eq2}
    \end{align}
    To verify~\eqref{eq:frechet-kolmogorov-eq2}, we use the fact that the second moments are uniformly bounded to
    obtain
    \begin{equation}
      \label{eq:coercivitysecond-moment-bound}
      \begin{split}
        \sup_k\int_{\mathbb{R}^d\setminus B_R}u_k(x)\,dx
        &\leq\frac{1}{R^2}\sup_k\int_{\mathbb{R}^d\setminus B_R}|x|^2 u_k(x)\,dx \\
        &\leq\frac{1}{R^2}\sup_{k}M_2(u_k) \\
        &\leq\frac{C}{R^2}\xrightarrow{R\uparrow\infty}0.
      \end{split}
    \end{equation}
    To verify~\eqref{eq:frechet-kolmogorov-eq1}, let $\eps>0$ and let $M=2$.
    For $R<\infty$ consider the decomposition
    \begin{equation}
      \label{eq:uk-split-3-terms}
      u_k=(u_k\wedge M+(u_k\vee M-M))\chi_{B_R}+u_k\chi_{\mathbb{R}^d\setminus B_R}.
    \end{equation}
    We treat those three terms separately.
    Let us start with the first one.
    The term $u_k\wedge M$ converges to $u\wedge M$ pointwise a.e., hence by dominated convergence we have
    $u_k\wedge M\rightarrow u\wedge M$ in $L^p_{loc}$ for all $p<\infty$, and we can choose $\delta_{1}>0$ small
    enough such that for all $|z|<\delta_{1}$ and all $k$
    \[\int_{\mathbb{R}^d}\left|\left((u_k\wedge M)\chi_{B_R}\right)(x-z)-\left((u_k\wedge M)\chi_{B_R}\right)(x)
      \right|\,dx<\eps/4.\]
    The last term can be made small by choosing $R<\infty$ sufficiently large.
    Precisely, we can choose a fixed $R<\infty$ such that for all $|z|<1$ and all $k$ we have
    \[\int_{\mathbb{R}^d}\left|\left((u_k\wedge M)\chi_{\mathbb{R}^{d}\setminus B_R}\right)(x-z)-\left((u_k\wedge
          M)\chi_{\mathbb{R}^{d}\setminus B_R}\right)(x)\right|\,dx
      \leq\frac{2}{R^{2}}\sup_{k}M_{2}(u_{k})
      <\eps/4.
    \]
    For the second term, let $\omega_R(z)\coloneqq\mathcal{L}^d(B_{R+2|z|}\setminus B_{R})$.
    Then, by the triangle inequality and since $\phi$ is monotonically increasing on $[1,\infty)$, we have
    \begin{align*}
      &\int_{\mathbb{R}^d}
        \left|\left((u_k\vee M-M)\chi_{B_R}\right)(x-z)-\left((u_k\vee M-M)\chi_{B_R}\right)(x)\right|\,dx \\
      \leq\,\,& M\omega_R(z)+\int_{\mathbb{R}^d}
                |\left((u_k\vee M)\chi_{B_R}\right)(x-z)-\left((u_k\vee M)\chi_{B_R}\right)(x)|\,dx \\
      \leq\,\,& M\omega_R(z)
                +\left(\Lip_{[M,\infty)}\phi^{-1}\right)\int_{\mathbb{R}^d}
                \left|\left((w_k\vee\phi(M))\chi_{B_R}\right)(x-z)
                -\left((w_k\vee\phi(M))\chi_{B_R}\right)(x)\right|\,dx.
    \end{align*}
    Note that since $M=2>1$, we have $\Lip_{[M,\infty)}\phi^{-1}<\infty$.
    Moreover, we have
    \begin{align*}
      &\int_{\mathbb{R}^d}\left|\left((w_k\vee\phi(M))\chi_{B_R}\right)(x-z)
        -\left((w_k\vee\phi(M))\chi_{B_R}\right)(x)\right|\,dx \\
      \leq\,\,& \|w_k\vee\phi(M)\chi_{B_{2R}}\|_{L^{2}}\left(\omega_{R}(z)\right)^{1/2}
                +\int_{B_{R+|z|}}|(w_k\vee\phi(M))(x-z)-(w_k\vee\phi(M))(x)|\,dx.
    \end{align*}
    The second term on the RHS converges to zero uniformly in $k$ as $|z|\rightarrow 0$,
    because $w_k$ converges in $L^1_{loc}$, and $x\mapsto w_k(x)\vee\phi(M)$ is Lipschitz.
    Thus we can choose $\delta_{2}>0$ such that for all $|z|<\delta_{2}$ and all $k$
    \[\left(\Lip_{[M,\infty)}\phi^{-1}\right)\int_{B_{R+|z|}}|(w_k\vee\phi(M))(x-z)-(w_k\vee\phi(M))(x)|\,
      dx<\eps/4.\]
    Finally, the term $\|w_k\vee\phi(M)\|_{L^{2}}$ is bounded by a uniform
    constant and $\omega_R$ is continuous with
    $\omega_R(z)\rightarrow 0$ as $|z|\downarrow 0$, thus we may choose
    $\delta_{3}>0$ such that for all $|z|<\delta_{3}$
    \[2\omega_{R}(z)+\left(\Lip_{[M,\infty)}\phi^{-1}\right)\|w_k\vee\phi(M)\|_{L^{2}}(\omega_R(z))^{1/2}<\eps/4.\]
    Therefore we can conclude that for all $|z|<\delta\coloneqq\min(\delta_{1},\delta_{2},\delta_{3})$
    by~\eqref{eq:uk-split-3-terms} and the triangle inequality,
    \[\sup_k\int_{\mathbb{R}^d}|u_k(x-z)-u_k(x)|\,dx<\eps.\]
  \end{step}
  \begin{step}[Lower semi-continuity]
    Let $u_{k}\rightarrow u$ in $L^{1}$ with $u_{k}\in D$.
    Since $D$ is closed, we have $u\in D$.
    By interpolation between $L^1$ and $L^4$, we get $u_k\rightarrow u$ in $L^3$
    for a subsequence, and
    \[w_{k}\coloneqq\phi\circ u_{k}\rightarrow\phi\circ u\eqqcolon w\quad\text{in $L^{1}$.}\]
    Moreover $\mathcal{F}_D$ can be represented as
    \[\mathcal{F}_D(w)
      =\int_{\mathbb{R}^d}|\nabla w|\,dx
      =\sup\{
      \int_{\mathbb{R}^d}(\nabla\cdot\xi)w\,dx:\xi\in C_{c}^{\infty}(\mathbb{R}^{d};\mathbb{R}^{d}),\,
      \|\xi\|_{\infty}\leq 1\},
    \]
    i.e., $\mathcal{F}_D$ can be written as the supremum of continuous linear
    functionals on $L^1(\mathbb{R}^d)$.
    Hence $\mathcal{F}_D$ is lower semi-continuous on $L^1(\mathbb{R}^d)$. \qedhere
  \end{step}
\end{proof}

\setcounter{step}{0}
\begin{proof}[Proof of
  Theorem~\ref{thm:u-eps-convergence-to-weak-hele-shaw-solution} (i) and (ii)]
  We are now in the position to prove the first two items
  of Theorem~\ref{thm:u-eps-convergence-to-weak-hele-shaw-solution}.
  The proof of the last item is given in the next subsection
  \begin{step}
    [Proof of (i)]
    We show that there exists a family of finite perimeter sets
    $(\Omega(t))_{t\in[0,T]}$ such that
    $u_{\eps}(\cdot,t)\rightarrow\chi_{\Omega(t)}$ in $L^1$ for a.e.\ $t\in[0,T]$.

    By Theorem~\ref{thm:rossi-savare-compactness-in-measure}, there exists a subsequence
    $u_{\eps_l}$ and $u$ such that $u_{\eps_l}\rightarrow u$ in $\mathcal{M}((0,T);L^1(\mathbb{R}^d))$.
    Then for a further subsequence we have for a.e. $t$
    \[\lim_{l\uparrow\infty}\|u_{\eps_l}(\cdot,t)-u(\cdot,t)\|_{L^1}=0.\]
    Then $u_{\eps_l}\rightarrow u$ in $L^1(0,T;L^1(\mathbb{R}^d))$ by
    dominated convergence.

    Now we show that $u(x,t)\in\{0,1\}$ a.e.\ in $\mathbb{R}^d$ for a.e. $t\in[0,T]$.
    By assumption, the energies are uniformly bounded by
    \[E_{\eps_l}(u_{\eps_l}(\cdot,t))
      =\int_{\mathbb{R}^d}\frac{\eps_l}{2}|\nabla u_{\eps_l}|^2
      +\frac{1}{\eps_l}W(u_{\eps_l}(\cdot,t))\,dx
      \leq E_0<\infty.
    \]
    Therefore, by Fatou's lemma
    \begin{align*}
      \int_{\mathbb{R}^d}W(u(x,t))\,dx
      &=\int_{\mathbb{R}^d}\liminf_{l\uparrow\infty} W(u_{\eps_l}(x,t))\,dx \\
      &\leq\lim_{l\uparrow\infty}
        \int_{\mathbb{R}^d}W(u_{\eps_l}(x,t))\,dx=0.
    \end{align*}
    Recall that $W(s)=\frac{1}{4}s^2(s-1)^2$, hence for a.e. $(x,t)\in\mathbb{R}^d\times [0,T]$ we have
    $u(x,t)\in\{s\in\mathbb{R}:W(s)=0\}=\{0,1\}$.
    Let $\Omega(t)\coloneqq\{x\in\mathbb{R}^d:\lim_{\eps_l\downarrow 0}u_{\eps_l}(x,t)=1\}$.

    It remains to show that $\Omega(t)$ has finite perimeter.
    By Fatou's Lemma we have for all $\xi\in C_c^1(\Omega(t);\mathbb{R}^d)$ with $|\xi|\leq 1$,
    \begin{align*}
      \int_{\mathbb{R}^d}\sigma\chi_{\Omega(t)}\div\xi\,dx
      &\leq\liminf_{l\uparrow\infty}
        \int_{\mathbb{R}^d}(\phi\circ u_{\eps_l})(x,t)\div\xi\,dx \\
      &\leq\liminf_{l\uparrow\infty}\int_{\mathbb{R}^d}|\nabla(\phi\circ u_{\eps_l})|(x,t)\,dx,
    \end{align*}
    where $\sigma=\phi(1)$.
    Now take the supremum over $\xi$ and use~\eqref{eq:modica-mortola-trick} to get
    \begin{equation}
      \label{eq:omega-t-perimeter-upper-bound-by-energy}
      \begin{split}
        \sigma P(\Omega(t))
        &=\sup\{\int_{\mathbb{R}^d}\sigma\chi_{\Omega(t)}\div\xi\,dx:
        \xi\in C_c^1(\Omega(t);\mathbb{R}^d),\,|\xi|\leq 1\} \\
        &\leq\liminf_{l\uparrow\infty}
        \int_{\mathbb{R}^d}|\nabla (\phi\circ u_{\eps_l})|(x,t)\,dx \\
        &\leq\liminf_{l\uparrow\infty}E_{\eps_l}(u_{\eps_l}(\cdot,t))<\infty.
      \end{split}
    \end{equation}
  \end{step}
  \begin{step}
    [Proof of (ii)]
    Using Cauchy-Schwarz and~\eqref{eq:optimal-energy-dissipation}, we get a uniform bound on $j_{\eps_l}$ in $L^1$:
    \begin{equation}
      \label{eq:L1-estimate-on-measure-u-eps}
      \begin{split}
        \int_0^T\int_{\mathbb{R}^d}|j_{\eps_l}|\,dx\,dt
        &\leq\left(\int_0^T\int_{\mathbb{R}^d}u_{\eps_l}\,dx\,dt\right)^{1/2}
        \left(\int_0^T\int_{\mathbb{R}^d}\frac{|j_{\eps_l}|^2}{u_{\eps_l}}\,dx\,dt\right)^{1/2} \\
        &=\sqrt{T}\left(\int_0^T\int_{\mathbb{R}^d}\frac{|j_{\eps_l}|^2}{u_{\eps_l}}\,dx\,dt\right)^{1/2} \\
        &\leq\sqrt{T}(2E_0)^{1/2} \\
        &=\sqrt{2TE_0}.
      \end{split}
    \end{equation}
    Then, by~\cite[Thm.\ 1.41]{Evans2015}, there exists a subsequence
    $j_{\eps_l}$ and a Radon measure $j$ such that
    \[j_{\eps_l}\xrightharpoonup{*} j\,\quad
      \text{weakly-* as Radon measures.}
    \]
    Now let $\Omega\coloneqq\bigcup_{t\in[0,T]}\Omega(t)\times\{t\}\subset\mathbb{R}^{d+1}$.
    We show that $\supp j\subset\overline\Omega$.
    Let $0<t_0<t_1<T$, and let $U\subset\mathbb{R}^d$ be open.
    We localize the estimate~\eqref{eq:L1-estimate-on-measure-u-eps} in time and space to obtain
    \begin{equation}
      \label{eq:local-estimate-on-measure-v}
      \begin{split}
        \int_{t_0}^{t_1}\int_{U}|j_{\eps_l}|\,dx\,dt
        &\leq\left(\int_{t_0}^{t_1}\int_{U}u_{\eps_l}\,dx\,dt\right)^{1/2}
        \left(\int_{t_0}^{t_1}\int_{U}\frac{|j_{\eps_l}|^2}{u_{\eps_l}}\,dx\,dt\right)^{1/2} \\
        &\leq\left(\int_{t_0}^{t_1}\int_{U}u_{\eps_l}\,dx\,dt\right)^{1/2}
        \left(\int_0^T\int_{\mathbb{R}^d}\frac{|j_{\eps_l}|^2}{u_{\eps_l}}\,dx\,dt\right)^{1/2} \\
        &\leq\left(\int_{t_0}^{t_1}\int_{U}u_{\eps_l}\,dx\,dt\right)^{1/2}(2E_0)^{1/2}.
      \end{split}
    \end{equation}
    Now passage to the limit $\eps\downarrow 0$ gives
    \[|j|(U\times(t_0,t_1))
      \leq\liminf_{l\uparrow\infty}
      \int_{t_0}^{t_1}\int_{U(t)}|j_{\eps_l}|\,dx\,dt
      \leq\sqrt{2E_0}\left( \int_{t_0}^{t_1}\int_{U}\chi_{\Omega(t)}\,dx\,dt \right)^{1/2}.
    \]
    If we choose $U$ such that $\mathcal{L}^{d+1}((U\times(t_0,t_1))\cap\overline{\Omega})=0$,
    then the RHS is zero.
    Now, let $x\in\mathbb{R}^d$ and let $t\in(0,T)$.
    If $(x,t)\in\mathbb{R}^{d+1}\setminus\overline\Omega$, there exist $0<t_0<t_1<T$ and $U\subset\mathbb{R}^d$
    such that $U\times(t_0,t_1)\subset\mathbb{R}^{d+1}\setminus\overline\Omega$.
    Then $|j|(U\times(t_0,t_1))=0$, and therefore
    \[\supp j\subset\overline\Omega.\]
    Again using joint convexity, by weak convergence of $j_{\eps_l}$ we have
    \begin{equation}
      \label{eq:sharp-interface-flux-L2-bound}
      \int_0^T\int_{\Omega(t)}|j|^2\,dx\,dt
      \leq\liminf_{l\rightarrow\infty}\int_0^T\int_{\mathbb{R}^d}\frac{|j_{\eps_l}|^2}{u_{\eps_l}}\,dx\,dt.
    \end{equation}
    In particular $j\in L^2(0,T;L^2(\mathbb{R}^d;\mathbb{R}^d))$. \qedhere
  \end{step}
\end{proof}

\subsection{Convergence}
\label{sec:sharp-interface-limit-convergence}

\setcounter{step}{0}
\begin{proof}[Proof of
  Theorem~\ref{thm:u-eps-convergence-to-weak-hele-shaw-solution} (iii)]
  The proof of item (iii) is divided into two steps.
  First, we verify~\eqref{eq:def-weak-sol-hele-shaw-conservation-eq}
  and~\eqref{eq:def-weak-sol-hele-shaw-curvature-eq}.
  In the second step, we prove the
  optimal dissipation inequality~\eqref{eq:optimal-energy-dissipation-eps-limit},
  \begin{step}
    To show that~\eqref{eq:def-weak-sol-hele-shaw-conservation-eq} holds,
    we pass to the limit in the corresponding equation for $u_{\eps_l}$.
    Let $\zeta\in C_c^1(\mathbb{R}^d\times [0,T))$,
    \[\int_{\mathbb{R}^d}u_{\eps_l,0}\zeta(\cdot,0)\,dx
      +\int_0^T\int_{\mathbb{R}^d}
      u_{\eps_l}\partial_t\zeta+j_{\eps_l}\cdot\nabla\zeta\,dx\,dt
      =0.
    \]
    Now the first term converges by~\eqref{eq:initial-data-well-prepared}.
    The second and third term converge by (i) and (ii) of
    Theorem~\ref{thm:u-eps-convergence-to-weak-hele-shaw-solution}.
    
    It remains to show~\eqref{eq:def-weak-sol-hele-shaw-curvature-eq}, i.e.,\
    for all $\xi\in C_c^1(\mathbb{R}^d\times(0,T);\mathbb{R}^d)$ with $\nabla\cdot\xi=0$ we have
    \begin{equation}
      \label{eq:curvature-operators-converge}
      \int_0^T\int_{\Omega(t)}j\cdot\xi\,dx\,dt
      =\int_0^T\int_{\mathbb{R}^d}\nabla\cdot\xi-\nu_\Omega\cdot\nabla\xi\,\nu_{\Omega}\,d|\mu_{\Omega}|\,dt,
    \end{equation}
    Here, $\nu_\Omega(\cdot,t):\partial^*\Omega(t)\rightarrow S^{d-1}$ is the
    measure theoretic outward unit normal and
    $\mu_{\Omega}(\cdot,t)=\mu_{\Omega(t)}$ is the Gauss-Green measure at time $t$.

    First, we combine~\eqref{eq:perimeter-bounded-from-below-by-energy}
    and~\eqref{eq:omega-t-perimeter-upper-bound-by-energy} to get
    \begin{equation}
      \label{eq:energy-strong-convergence-to-perimeter}
      \lim_{l\uparrow\infty}\int_0^T E_{\eps_l}(u_{\eps_l}(\cdot,t))\,dt=\sigma\int_0^T P(\Omega(t))\,dt.
    \end{equation}
    Let $\xi\in C_c^2(\mathbb{R}^d\times(0,T);\mathbb{R}^d)$ such that $\nabla\cdot\xi=0$,
    and let $\boldsymbol{T}_{\eps}$ as in~\eqref{eq:def-T-eps}, i.e.,
    \[\boldsymbol{T}_{\eps}\coloneqq\left(\frac{\eps}{2}|\nabla u_{\eps}|^2+\frac{1}{\eps}W(u_{\eps})\right)I_d
      -\eps\nabla u_{\eps}\otimes\nabla u_{\eps}.\]
    We want to derive~\eqref{eq:curvature-operators-converge} by passing to the limit
    in~\eqref{eq:weak-sol-cahn-hilliard-eq-2}, i.e.,\
    \[\int_0^T\int_{\mathbb{R}^d}j_{\eps_l}\cdot\xi\,dx\,dt
      =\int_0^T\int_{\mathbb{R}^d}\boldsymbol{T}_{\eps_l}:\nabla\xi\,dx\,dt.
    \]
    Since $j_{\eps_l}\xrightharpoonup{*}\chi_\Omega j$, it suffices to show that
    \begin{equation}
      \label{eq:curvature-operators-converge-2-T-eps}
      \boldsymbol{T}_{\eps_l}\,dx\,dt
      \xrightharpoonup{*}\sigma\left(I_d-\nu_\Omega\otimes\nu_\Omega\right)\,d|\mu_{\Omega}|\,dt\quad
      \text{weakly-* as Radon measures.}
    \end{equation}
    For the first term, we note that it is sufficient
    to test $\frac{\eps_l}{2}|\nabla u_{\eps_l}|^2+\frac{1}{\eps_l}W(u_{\eps_l})$
    with $\zeta\in C_0(\mathbb{R}^d\times(0,T))$ such that $0\leq\zeta\leq 1$.
    We have, by~\eqref{eq:modica-mortola-young-ineq},
    \begin{align*}
      \int_0^T\int_{\mathbb{R}^d}\zeta\left( \frac{\eps_l}{2}|\nabla u_{\eps_l}|^2+\frac{1}{\eps_l}W(u_{\eps_l})
      \right)\,dx\,dt
      &\geq\int_0^T\int_{\mathbb{R}^d}\zeta|\nabla(\phi\circ u_{\eps_l})|\,dx\,dt.
    \end{align*}
    Hence we have
    \[\liminf_{l\uparrow\infty}
      \int_0^T\int_{\mathbb{R}^d}\zeta\left(
        \frac{\eps_l}{2}|\nabla u_{\eps_l}|^2+\frac{1}{\eps_l}W(u_{\eps_l})
      \right)\,dx\,dt
      \geq\sigma\int_0^T\int_{\mathbb{R}^d}\zeta\,d|\mu_\Omega|\,dt.
    \]
    Now we test with $\eta=1-\zeta$ to obtain
    \[\liminf_{l\uparrow\infty}
      \int_0^T\int_{\mathbb{R}^d}(1-\zeta)\left(
        \frac{\eps_l}{2}|\nabla u_{\eps_l}|^2+\frac{1}{\eps_l}W(u_{\eps_l})
      \right)\,dx\,dt
      \geq\sigma\int_0^T\int_{\mathbb{R}^d}(1-\zeta)\,d|\mu_\Omega|\,dt.
    \]
    Combining these two inequalities with~\eqref{eq:energy-strong-convergence-to-perimeter}, we get
    \[\lim_{l\uparrow\infty}
      \int_0^T\int_{\mathbb{R}^d}\zeta\left(
        \frac{\eps_l}{2}|\nabla u_{\eps_l}|^2+\frac{1}{\eps_l}W(u_{\eps_l})
      \right)\,dx\,dt
      =\sigma\int_0^T\int_{\mathbb{R}^d}\zeta\,d|\mu_\Omega|\,dt.
    \]
    This gives us
    \[\left(
        \frac{\eps_l}{2}|\nabla u_{\eps_l}|^2+\frac{1}{\eps_l}W(u_{\eps_l})
      \right)\,dx\,dt\xrightharpoonup{*}\sigma|\mu_\Omega|\,dt.
    \]
    and
    \begin{equation}
      \label{eq:total-variation-convergence-gauss-green}
      |\nabla(\phi\circ u_{\eps_l})|\,dx\,dt\xrightharpoonup{*}\sigma|\mu_\Omega|\,dt.
    \end{equation}
    For the second term, we define
    \[\nu_{\eps_l}=-\frac{\nabla(\phi\circ u_{\eps_l})}{|\nabla(\phi\circ u_{\eps_l})|},\]
    with the convention $\nu_{\eps_l}=e_1$ if $\nabla(\phi\circ u_{\eps_l})=0$.
    Let $\nu^*\in C_c^1(\mathbb{R}^d\times (0,T);\mathbb{R}^d)$ and let
    \begin{equation}
      \label{eq:def-theta-eps-and-theta}
      \begin{split}
        \mathcal{E}_{\eps_l}(u_{\eps_l};v^*)
        &\coloneqq\int_0^T\int_{\mathbb{R}^d}
        |\nu_{\eps_l}-\nu^*|^2|\nabla(\phi\circ u_{\eps_l})|\,dx\,dt, \\
        \mathcal{E}(\Omega;\nu^*)
        &\coloneqq\sigma\int_0^T\int_{\mathbb{R}^d}
        |\nu_\Omega-\nu^*|^2\,d|\mu_\Omega|\,dt.
      \end{split}
    \end{equation}
    Using the fact that
    $-\frac{\nabla u_{\eps_l}}{|\nabla u_{\eps_l}|}=\nu_{\eps_l}$ a.e. we compute
    \begin{align*}
      &\left|
        \int_0^T\int_{\mathbb{R}^d}\nu_{\eps_l}\cdot\nabla\xi\nu_{\eps_l}|\nabla(\phi\circ u_{\eps_l})|\,dx\,dt
        -\sigma\int_0^T\int_{\mathbb{R}^d}\nu_{\Omega}\cdot\nabla\xi\,\nu_{\Omega}\,d|\mu_{\Omega}|\,dt
        \right| \\
      \leq\,\,&
                \left|\int_0^T\int_{\mathbb{R}^d}\nu^*\cdot\nabla\xi\nu_{\eps_l}|\nabla(\phi\circ u_{\eps_l})|\,dx\,dt
                -\sigma\int_0^T\int_{\mathbb{R}^d}\nu^*\cdot\nabla\xi\,\nu_{\Omega}\,d|\mu_{\Omega}|\,dt
                \right| \\
      &+\left|\int_0^T\int_{\mathbb{R}^d}
        (\nu_{\eps_l}-\nu^*)\cdot\nabla\xi\,\nu_{\eps_l}|\nabla(\phi\circ u_{\eps_l})|\,dx\,dt
        \right| \\
      &+\left|\int_0^T\int_{\mathbb{R}^d}(\nu_\Omega-\nu^*)\cdot\nabla\xi\,\nu_{\Omega}\,d|\mu_{\Omega}|\,
        dt\right| \\
      \leq\,\,&\left|\int_0^T\int_{\mathbb{R}^d}\nu^*\cdot\nabla\xi\cdot\nabla(\phi\circ u_{\eps_l})\,dx\,dt
                -\sigma\int_0^T\int_{\mathbb{R}^d}\nu^*\cdot\nabla\xi\,\nu_{\Omega}\,d|\mu_{\Omega}|\,dt
                \right| \\
      &+\left(\int_0^T\int_{\mathbb{R}^d}|\nabla\xi|^2|\nabla(\phi\circ u_{\eps_l})|\,dx\,dt\right)^{1/2}
        \mathcal{E}_{\eps_l}^{1/2}(\nu^*) \\
      &+\left(\int_0^T\int_{\mathbb{R}^d}|\nabla\xi|^2\,d|\mu_\Omega|\,dt\right)^{1/2}
        \mathcal{E}^{1/2}(\nu^*).
    \end{align*}
    By Lemma~\ref{lem:theta-eps-to-theta} below,
    \[\mathcal{E}_{\eps_l}(u_{\eps_l};\nu^*)\rightarrow\mathcal{E}(\Omega;\nu^*).\]
    Note further that $\nu^*\cdot\nabla\xi$ is an admissible test function in the weak convergence of
    $\nu_{\eps_l}\,dx\,dt\xrightharpoonup{*}\nu_\Omega\,|\mu_\Omega|\,dt$.
    Therefore, first letting $l\uparrow\infty$ and then $\nu^*\rightarrow\nu_\Omega$ we obtain
    \[\lim_{l\uparrow\infty}\left|
        \int_0^T\int_{\mathbb{R}^d}\nu_{\eps_l}\cdot\nabla\xi\,\nu_{\eps_l}|\nabla(\phi\circ u_{\eps_l})|\,dx\,dt
        -\sigma\int_0^T\int_{\mathbb{R}^d}\nu_{\Omega}\cdot\nabla\xi\,\nu_{\Omega}\,d|\mu_{\Omega}|\,dt
      \right|
      =0.
    \]
    \end{step}
    \begin{step}[Optimal energy dissipation]
      Using that $u_{\eps_l}\rightarrow\chi_{\Omega}$ in
      $L^2(0,T;L^1(\mathbb{R}^d))$ and~\eqref{eq:sharp-interface-flux-L2-bound},
      we have for a.e.\ $T'\in[0,T]$
      \begin{align*}
        \sigma P(\Omega(T'))+\int_0^{T'}\int_{\Omega(t)}|j|^2\,dx\,dt
        &\leq\liminf_{l\uparrow\infty}\left( E_{\eps_l}(u_{\eps_l}(\cdot,T'))
          +\int_{0}^{T'}\int_{\mathbb{R}^d}\frac{|j_{\eps_l}|^2}{u_{\eps_l}}\,dx\,dt
          \right) \\
        &\leq\liminf_{l\uparrow\infty}E_{\eps_l}(u_{\eps_l,0}) \\
        &=\sigma P(\Omega_0). \qedhere
      \end{align*}
    \end{step}
\end{proof}

\begin{lemma}
  \label{lem:theta-eps-to-theta}
  Assume $u_{\eps_l}\rightarrow\chi_\Omega$ in $L^1$ and
  $\int_0^T E_{\eps_l}(u_{\eps_l}(\cdot,t))\,dt\rightarrow\sigma\int_0^T P(\Omega(t))\,dt$.
  For $\nu^*\in C_c^1(\mathbb{R}^d\times(0,T);\mathbb{R}^d)$ let
  $\mathcal{E}_{\eps_l}(u_{\eps_l};\nu^*)$ and $\mathcal{E}(\Omega;\nu^*)$ as
  in~\eqref{eq:def-theta-eps-and-theta}.
  Then
  \[\lim_{l\uparrow\infty}\mathcal{E}_{\eps_l}(u_{\eps_l};\nu^*)=\mathcal{E}(\Omega;\nu^*).\]
\end{lemma}

\begin{proof}
  We expand the square $|\nu_{\eps_l}-\nu^*|^2=1+|\nu^*|^2-2\nu^*\cdot\nu_{\eps_l}$
  and integrate the last term by parts
  \begin{align*}
    \mathcal{E}_{\eps_l}(u_{\eps_l};\nu^*)
    &=\int_0^T\int_{\mathbb{R}^d}(1+|\nu^*|^2)|\nabla(\phi\circ u_{\eps_l})|\,dx\,dt
      -2\int_0^T\int_{\mathbb{R}^d}(\phi\circ u_{\eps_l})\nabla\cdot\nu^*\,dx\,dt, \\
    \mathcal{E}(\Omega;\nu^*)
    &=\sigma\int_0^T\int_{\mathbb{R}^d}(1+|\nu^*|^2)\,d|\mu_\Omega|\,dt
      -2\sigma\int_0^T\int_{\mathbb{R}^d}\nu_\Omega\cdot\nu^*\,d|\mu_\Omega|\,dt.
  \end{align*}
  The first terms converge by
  ~\eqref{eq:total-variation-convergence-gauss-green}.
  The last term in the second line reads,
  by definition of the Gauss-Green measure $\mu_\Omega$,
  \[2\sigma\int_0^T\int_{\mathbb{R}^d}\nu_\Omega\cdot\nu^*\,d|\mu_\Omega|\,dt
    =2\sigma\int_0^T\int_{\mathbb{R}^d}\chi_\Omega\nabla\cdot\nu^*\,dx\,dt.
  \]
  Now by the $L^1$ convergence
  $\phi\circ u_{\eps_l}\rightarrow\sigma\chi_\Omega$, the claim follows.
\end{proof}

The next lemma shows that the energy contribution of both summands
in the energy density is essentially the same as $\eps\downarrow 0$.
This was first shown by Luckhaus and Modica in~\cite{Luckhaus1989}.
For the convenience of the reader we recall the proof here.

\begin{lemma}[{\cite[Lemma 1]{Luckhaus1989}}]
  \label{lem:equi-participation}
  Let $a_{\eps_l}=(\eps_l/2)^{1/2}|\nabla u_{\eps_l}|$ and $b_{\eps_l}=\eps_l^{-1/2}W^{1/2}(u_{\eps_l})$.
  Further let $\{\Omega(t)\}_{t\in[0,T]}$ be a family of finite perimeter sets such
  that~\eqref{eq:u-eps-L1-convergence} holds.
  If
  \begin{equation}
    \label{eq:lem-energy-conv-assumption}
    \int_0^T E_{\eps_l}(u_{\eps_l}(\cdot,t))\,dt\rightarrow\sigma\int_0^T P(\Omega(t))\,dt,
  \end{equation}
  then
  \begin{equation}
    \begin{split}
      \liminf_{l\uparrow\infty}\int_0^T\int_{\mathbb{R}^d}(a_{\eps_l}-b_{\eps_l})^2\,dx\,dt
      &=\liminf_{l\uparrow\infty}\int_0^T\int_{\mathbb{R}^d}|a_{\eps_l}^2-a_{\eps_l}b_{\eps_l}|\,dx\,dt \\
      &=\liminf_{l\uparrow\infty}\int_0^T\int_{\mathbb{R}^d}|a_{\eps_l}^2-b_{\eps_l}^2|\,dx\,dt
      =0.
    \end{split}
  \end{equation}
\end{lemma}

\begin{proof}
  On the one hand, by~\eqref{eq:lem-energy-conv-assumption},
  \[\lim_{l\uparrow\infty}\int_0^T\int_{\mathbb{R}^d}a_{\eps_l}^2+b_{\eps_l}^2\,dx\,dt
    =\sigma\int_0^T P(\Omega(t))\,dt.
  \]
  On the other hand, by~\eqref{eq:total-variation-convergence-gauss-green}
  and lower semi-continuity,
  \[\liminf_{l\uparrow\infty}\int_0^T\int_{\mathbb{R}^d}2 a_{\eps_l}b_{\eps_l}\,dx\,dt
    =\liminf_{l\uparrow\infty}\int_0^T\int_{\mathbb{R}^d}|\nabla(\phi\circ u_{\eps_l})|\,dx\,dt
    \geq\sigma\int_0^T P(\Omega(t))\,dt.
  \]
  Therefore
  \[\liminf_{l\uparrow\infty}\int_0^T\int_{\mathbb{R}^d}(a_{\eps_l}-b_{\eps_l})^2\,dx\,dt=0.\]
  Moreover, by the H{\"o}lder inequality and the uniform bound on $a_{\eps_l}$ in $L^2$, we have
  \begin{align*}
    \int_0^T\int_{\mathbb{R}^d}|a_{\eps_l}^2-a_{\eps_l}b_{\eps_l}|\,dx\,dt
    &\leq\left(\int_0^T\int_{\mathbb{R}^d}|a_{\eps_l}|^2\,dx\,dt\right)^{1/2}
      \left(\int_0^T\int_{\mathbb{R}^d}|a_{\eps_l}-b_{\eps_l}|^2\,dx\,dt\right)^{1/2} \\
    &\leq C\left(\int_0^T\int_{\mathbb{R}^d}(a_{\eps_l}-b_{\eps_l})^2\,dx\,dt\right)^{1/2}.
  \end{align*}
  The same argument applies for $|b_{\eps_l}^2-a_{\eps_l}b_{\eps_l}|$.
  Finally, observe that
  \begin{align*}
    \int_0^T\int_{\mathbb{R}^d}|a_{\eps_l}^2-b_{\eps_l}^2|\,dx\,dt
    &\leq\int_0^T\int_{\mathbb{R}^d}|a_{\eps_l}^2-a_{\eps_l}b_{\eps_l}|\,dx\,dt
      +\int_0^T\int_{\mathbb{R}^d}|b_{\eps_l}^2-a_{\eps_l}b_{\eps_l}|\,dx\,dt.
  \end{align*}
  This concludes the proof.
\end{proof}

\subsection{Weak solutions are strong solutions}
\label{sec:weak-sol-are-strong-sol-hele-shaw}

Now we show that, if $j$ and
$\Omega\coloneqq\bigcup_{t\in[0,T]}\Omega(t)\times\{t\}$
are smooth and a weak solution to the Hele--Shaw flow,
then $\Omega$ and $j$ solve the Hele--Shaw equations in the classical sense,
that is, $\Omega$ and $j$ satisfy~\eqref{eq:velocity-eq}--\eqref{eq:conservation-eq}.

\begin{lemma}
  \label{lem:weak-sol-is-classical-sol-if-smooth}
  Let $(\Omega,j)$ be a weak solution to the Hele--Shaw flow in the sense of Definition~\ref{def:weak-sol-hele-shaw}.
  If $j$ is smooth and $\Omega(t)$ evolves smoothly and is simply connected for all $t$,
  then $(\Omega,j)$ is a classical solution to the Hele--Shaw
  flow~\eqref{eq:velocity-eq}--\eqref{eq:conservation-eq}.
\end{lemma}

\setcounter{step}{0}
\begin{proof}
  \begin{step}[$(\Omega,j)$ solves~\eqref{eq:velocity-eq}]
    By~\eqref{eq:def-weak-sol-hele-shaw-curvature-eq}, we have
    \begin{equation}
      \label{eq:weak-limit-eq}
      \int_0^T\int_{\mathbb{R}^d}
      \chi_{\Omega(t)}\partial_t\zeta+\chi_{\Omega(t)}j\cdot\nabla\zeta\,dx\,dt
      =0
    \end{equation}
    for all $\zeta\in C_c^1(\mathbb{R}^d\times(0,T))$.
    Now, let $t_0\in(0,T)$ and $x_0\in\mathrm{Int}(\Omega(t_0))$, the interior
    of $\Omega(t_0)$.
    Since $\Omega(t)$ evolves smoothly, there exists $\eps>0$ such that
    \[B_\eps(x_0)\subset\Omega(t)
      \quad\text{for all $t\in(t_0-\eps,t_0+\eps)$.}
    \]
    Let $\zeta\in C_c^1(B_\eps(x_0)\times(t_0-\eps,t_0+\eps))$ and let
    $\phi(t)\coloneqq\int_{B_\eps(x_0)}\zeta(x,t)\,dx$.
    Then $\phi(t_0-\eps)=\phi(t_0+\eps)=0$.
    Using differentiation under the Integral, we have
    \[\int_{\Omega(t)}\partial_t\zeta(x,t)\,dx
      \int_{B_\eps(x_0)}\partial_t\zeta(x,t)\,dx
      =\frac{d}{dt}\int_{B_\eps(x_0)}\zeta(x,t)\,dx.
    \]
    The RHS vanishes after integration over $(t_0-\eps,t_0+\eps)$.
    By~\eqref{eq:weak-limit-eq} we have
    \begin{align*}
      0&=\int_0^T\int_{\mathbb{R}^d}
         \chi_{\Omega(t)}\partial_t\zeta+\chi_{\Omega(t)}j\cdot\nabla\zeta\,dx\,dt \\
       &=\int_{t_0-\eps}^{t_0+\eps}\frac{d}{dt}\phi(t)\,dt
         +\int_{t_0-\eps}^{t_0+\eps}\int_{B_\eps(x_0)}j\cdot\nabla\zeta\,dx\,dt \\
       &=\int_{t_0-\eps}^{t_0+\eps}\int_{B_\eps(x_0)}j\cdot\nabla\zeta\,dx\,dt.
    \end{align*}
    Therefore
    \[\nabla\cdot j(x_0,t_0)=0\quad\text{for all $(x_0,t_0)\in\Omega$.}\]
    It remains to show that $V=j\cdot\nu$.

    Let $V(t):\partial\Omega(t)\rightarrow\mathbb{R}^d$ be the normal velocity of $\partial\Omega(t)$ at time $t$
    and let $\zeta\in C_c^1
    (\mathbb{R}^d\times(0,T))$.
    Then, by~\cite[Appendix C4,\ Thm.\ 6]{Evans2010},
    \[0=\int_0^T\int_{\Omega(t)}\partial_t\zeta\,dx\,dt
      +\int_0^T\int_{\partial\Omega(t)}\zeta V\,d\mathcal{H}^{d-1}\,dt,
    \]
    and by~\eqref{eq:weak-limit-eq} we know that
    \[0=\int_0^T\int_{\Omega(t)}(-\partial_t\zeta-j\cdot\nabla\zeta)\,dx\,dt.
    \]
    Subtracting these two identities, using $\nabla\cdot j=0$ in $\Omega$ and Stokes' Theorem, we get
    \begin{equation}
      \label{eq:velocity-eq-stokes-thm}
      \begin{split}
        0&=\int_0^T\int_{\partial\Omega(t)}\zeta V\,d\mathcal{H}^{d-1}\,dt
        -\int_0^T\int_{\Omega(t)}\nabla\cdot(\zeta j)\,dx\,dt \\
        &=\int_0^T\int_{\partial\Omega(t)}\zeta V\,d\mathcal{H}^{d-1}\,dt
        -\int_0^T\int_{\partial\Omega(t)}\zeta j\cdot\nu\,d\mathcal{H}^{d-1}\,dt \\
        &=\int_0^T\int_{\partial\Omega(t)}\zeta(V-j\cdot\nu)\,d\mathcal{H}^{d-1}\,dt.
      \end{split}
    \end{equation}
    By the fundamental lemma of the calculus of variations, we have $V-j\cdot\nu=0$ on $\partial\Omega(t)$ for
    all $t$,
    since~\eqref{eq:velocity-eq-stokes-thm} holds for all $\zeta\in C_c^1(\mathbb{R}^d\times(0,T))$.
    Hence $j$ solves~\eqref{eq:velocity-eq}.
  \end{step}
  \begin{step}[$(\Omega,j)$ solves~\eqref{eq:conservation-eq}]
    We want to show that there exists a function
    $p:\Omega\rightarrow\mathbb{R}$ such that
    \[\begin{cases}
        j(\cdot,t)=-\nabla p(\cdot,t),\qquad&\text{in $\Omega(t)$,} \\
        p(\cdot,t)=\sigma H,&\text{on $\partial\Omega(t)$.}
      \end{cases}
    \]
    Fix $t\in[0,T]$ and let $\xi\in C_c^1(\Omega(t),\mathbb{R}^d)$ such that
    $\nabla\cdot\xi=0$.
    By~\eqref{eq:def-weak-sol-hele-shaw-curvature-eq} and since $\partial\Omega(t)$ is smooth,
    $j$ is the velocity of $\Omega$ and $\Omega$ is smooth, i.e.,
    \begin{equation}
      \label{eq:hele-shaw-eq-v-is-grad}
      \int_{\Omega(t)}\xi\cdot j(\cdot,t)\,dx
      =\int_{\partial\Omega(t)}(\nabla\cdot\xi-\nu\cdot\nabla\xi\nu)\,d\mathcal{H}^{d-1}
      =\int_{\partial\Omega(t)}\xi\cdot H\nu\,d\mathcal{H}^{d-1}.
    \end{equation}
    Since $\xi$ has compact support inside $\Omega(t)$, the RHS is zero, and hence
    \[\int_{\Omega(t)}\xi\cdot j(\cdot,t)\,dx=0.\]
    Thus $j(\cdot,t)\perp_{L^2}\{\nabla\cdot\xi=0\}$, that is, $j(\cdot,t)$ is perpendicular
    to the set of divergence free vector fields on $\Omega(t)$ w.r.t.\ the $L^2$-norm for all $t$.
    By the Helmholtz-decomposition~\cite[Chap.\ XIX, \S 1, Thm.\ 5]{Dautray1993},
    since $\Omega(t)$ is simply connected and smooth, $j$ is a gradient, that is,
    there exists a function $p:\Omega\rightarrow\mathbb{R}$ such that
    \begin{equation}
      \label{eq:v-is-grad-helmholtz-decomp}
      j=-\nabla p\quad\text{in $\Omega(t)$ for all $t$.}
    \end{equation}
    Now, let $\xi\in C_c^1(\mathbb{R}^d;\mathbb{R}^d)$ with $\nabla\cdot\xi=0$ and
    plug~\eqref{eq:v-is-grad-helmholtz-decomp}
    into~\eqref{eq:hele-shaw-eq-v-is-grad} to obtain
    \begin{align*}
      \int_{\Omega(t)}\xi\cdot\nabla p(\cdot,t)\,dx
      &=\sigma\int_{\partial\Omega(t)}\xi\cdot H\nu\,d\mathcal{H}^{d-1}.
    \end{align*}
    Using Stokes' Theorem and the fact that $\nabla\cdot\xi=0$, we also have
    \[\int_{\Omega(t)}\xi\cdot\nabla p(\cdot,t)\,dx
      =\int_{\Omega(t)}\nabla\cdot(p(\cdot,t)\xi)\,dx
      =\int_{\partial\Omega(t)}p(\cdot,t)\xi\cdot\nu\,d\mathcal{H}^{d-1}.
    \]
    Therefore
    \[\int_{\partial\Omega(t)}(p(\cdot,t)-\sigma H)\xi\cdot\nu\,d\mathcal{H}^{d-1}=0.\]
    Again, by the fundamental Lemma of the calculus of variations, $p(\cdot,t)-H=0$ on
    $\partial\Omega(t)$ for all $t$. \qedhere
  \end{step}
\end{proof}

\appendix

\section{Recap on optimal transport}
\label{sec:recap-on-optimal-transport}

We recall the quadratic optimal transport problem in the Euclidean setting.
Let $\mathcal{P}(\mathbb{R}^d)$ denote the space of probability measures on
$\mathbb{R}^d$.
Given $\mu,\nu\in\mathcal{P}(\mathbb{R}^d)$,
Monge's formulation asks for an optimal transport map
$\boldsymbol{t}$ which minimizes
\begin{equation}
  \label{eq:monge-formulation}
  \inf\{\int_{\mathbb{R}^d}|x-\boldsymbol{t}(x)|^2\,d\mu(x):\boldsymbol{t}_\#\mu=\nu\}.
\end{equation}
On the other hand, Kantorovich's formulation looks at all measures $\gamma$ with
marginals $\mu$ and $\nu$ and asks for a minimizer to
\begin{equation}
  \label{eq:kantorovich-formulation}
  \inf\{\int_{\mathbb{R}^d\times\mathbb{R}^d}|x-y|^2\,d\gamma(x,y):\gamma\in\Gamma(\mu,\nu)\},
\end{equation}
where
\[\Gamma(\mu,\nu)\coloneqq\{
  \gamma\in\mathcal{P}(\mathbb{R}^d\times\mathbb{R}^d):
  (\pi_x)_\sharp\gamma=\mu,\,(\pi_y)_\sharp\gamma=\nu
  \},
\]
and $\pi_x,\pi_y:\mathbb{R}^d\times\mathbb{R}^d\rightarrow\mathbb{R}^d$
denote the projections onto the first and second factor, respectively.
We observe that Kantorovich's formulation~\eqref{eq:kantorovich-formulation}.
is exactly the definition of the squared Wasserstein distance $W_2^2(\mu,\nu)$.
Moreover, it is well known that the squared Wasserstein distance satisfies
\begin{equation}
  \label{eq:wasserstein-distance-satisfies-monge-formulation}
  W_2^2(\mu,\nu)
  =\inf\{\int_{\mathbb{R}^d}|x-\boldsymbol{t}(x)|^2\,d\mu(x):\boldsymbol{t}_\#\mu=\nu\},
\end{equation}
whenever the RHS is well-posed, which is not always the case,
e.g.\ when $\mu$ is a Dirac mass but $\nu$ is not.

The following proposition gives a characterization of optimal transport plans
for absolutely continuous measures w.r.t.\ the Lebesgue measure.
For $1\leq p<\infty$ let
\[
  \mathcal{P}_p\coloneqq\{
  \mu\in\mathcal{P}(\mathbb{R}^d):
  \int_{\mathbb{R}^d}|x|^p\,d\mu(x)<\infty
  \},\quad
  \mathcal{P}_p^a\coloneqq\{
  \mu\in\mathcal{P}_p(\mathbb{R}^d):
  \mu=u\mathcal{L}^d
  \}.
\]

\begin{proposition}[Existence of optimal transport maps~{\cite[Thm.\ 2.3]{Ambrosio2007}}]
  \label{prop:existence-of-optimal-transport-maps}
  For any $\mu,\nu\in\mathcal{P}_2^a(\mathbb{R}^d)$,
  Kantorovich's optimal transport problem has a unique solution $\gamma$.
  Moreover:
  \begin{enumerate}[(i)]
  \item $\gamma$ is induced by a transport map $\boldsymbol{t}$, i.e.,
    $\gamma=(\boldsymbol{i}_d,\boldsymbol{t})_\sharp\mu$, where
    $\boldsymbol{i}_d$ is the identity map on $\mathbb{R}^d$.
    In particular, $\boldsymbol{t}$ is the unique solution of Monge's optimal
    transport problem~\eqref{eq:monge-formulation}.
  \item The map $\boldsymbol{t}$ coincides $\mu$-a.e.\ with the gradient of a convex
    function $\varphi:\mathbb{R}^d\rightarrow(-\infty,\infty]$, whose finiteness
    domain $D(\varphi)$ has non-empty interior and satisfies
    \[\mu(\mathbb{R}^d\setminus D(\varphi))
      =\mu(\mathbb{R}^d\setminus D(\nabla\varphi))=0.
    \]
  \item If $\boldsymbol{s}$ is the optimal transport map between $\nu$ and $\mu$,
    then
    \[\boldsymbol{s}\circ\boldsymbol{t}=\boldsymbol{i}_d
      \quad\text{$\mu$-a.e.\ in $\mathbb{R}^d$ and}\quad
      \boldsymbol{t}\circ\boldsymbol{s}=\boldsymbol{i}_d
      \quad\text{$\nu$-a.e.\ in $\mathbb{R}^d$.}
    \]
  \end{enumerate}
\end{proposition}

\begin{proposition}[{\cite[Thm.\ 7.2.2]{Ambrosio2008}}]
  \label{prop:characterization-constant-speed-geodesics}
  Let $\mu,\nu\in\mathcal{P}_2(\mathbb{R}^d)$ and
  $\gamma\in\Gamma(\mu,\nu)$.
  Then the map
  \[t\mapsto\mu_t\coloneqq((1-t)\pi_x+t\pi_y)_\sharp\gamma\]
  is a constant speed geodesic.
  Conversely, any constant speed geodesic
  $\mu_t:[0,1]\rightarrow\mathcal{P}_2(\mathbb{R}^d)$
  joining $\mu$ and $\nu$ has this representation for a suitable $\gamma\in\Gamma(\mu,\nu)$.
\end{proposition}

Here we say that $\mu_t$ is a constant speed geodesic, if
\[W_2(\mu_s,\mu_t)=(t-s)W_2(\mu_0,\mu_1)\]
whenever $0\leq s\leq t\leq 1$.

\begin{proposition}
  [{\cite[Thm.\ 8.3.1]{Ambrosio2008}}]
  \label{prop:constant-speed-geodesics-tangent-vector}
  Let $\mu_t:[0,1]\rightarrow\mathcal{P}_p^a(\mathbb{R}^d)$ be a constant speed
  geodesic.
  Then there exists a Borel vector field $v:(x,t)\mapsto v_t(x)$ such that
  \[
    v_t\in L^p(\mu_t,\mathbb{R}^d;\mathbb{R}^d)\quad
    \text{for $\mathcal{L}^1$-a.e.\ $t\in[0,1]$,}
  \]
  and the continuity equation
  \[\partial_t\mu_t+\nabla\cdot(v_t\mu_t)=0\]
  holds in the sense of distributions, i.e.,
  \begin{equation}
    \label{eq:continuity-eq-in-distributions}
    \int_0^1\int_{\mathbb{R}^d}\partial_t\zeta(x,t)+v_t(x)\cdot\nabla\zeta(x,t)\,d\mu_t(x)\,dt
    \qquad\text{for all $\zeta\in C_c^1(\mathbb{R}^d\times(0,1))$.}
  \end{equation}
\end{proposition}
Here
\[L^p(\mu,\mathbb{R}^d;\mathbb{R}^{d'})
  \coloneqq\{u:\mathbb{R}^d\rightarrow\mathbb{R}^{d'}:\int_{\mathbb{R}^d}|u|^p\,d\mu<\infty\}.
\]

\section{Construction of well-prepared initial data}\label{sec:recovery-sequence}

In this section we construct initial data which satisfies
the well-preparedness condition~\eqref{eq:initial-data-well-prepared}.

The construction is based on a famous result, which was first published by Luciano Modica and Stefano
Mortola~\cite{Modica1977} in 1977,
to construct suitable initial conditions in order to recover a solution to the Hele--Shaw flow
for a given initial configuration $\Omega_0$.

Luciano Modica and Stefano Mortola have shown that the functionals defined on $L^2(\mathbb{R}^d)$ by
\[E_\eps(u)=\int_{\mathbb{R}^d}\frac{\eps}{2}|\nabla u|^2+\frac{1}{\eps}W(u)\,dx
\]
$\Gamma$-converge to the functional
\[E_0(u)\coloneqq
  \begin{cases}
    \sigma P(\Omega;\mathbb{R}^d),\quad&\text{if $u=\chi_\Omega$,} \\
    \infty,&\text{else.}
  \end{cases}
\]
The following lemma is a slight variant of their recovery sequence and provides
well-prepared initial conditions.

\begin{lemma}
  Let $\Omega_0\subset\mathbb{R}^d$ be open, bounded, with $C^2$-boundary,
  $\mathcal{L}^d(\Omega_0)=1$.
  Then there exists a sequence
  $(\tilde{u}_{\eps,0})_{\eps>0}$ such that $\tilde{u}_{\eps,0}\in\mathcal{A}$ with
  $\sup_\eps M_2(\tilde{u}_{\eps,0})<\infty$ and
  \[\limsup_{\eps\downarrow 0}E_\eps(\tilde{u}_{\eps,0})\leq \sigma E_0(\chi_{\Omega_0}).
  \]
\end{lemma}

\setcounter{step}{0}
\begin{proof}
  The idea is simply to scale the recovery sequence of Modica-Mortola to ensure
  $\int\tilde u_{\eps,0}\,dx=1$.
  Then we only need to check that the second moments are uniformly bounded.
  \begin{step}[Optimal profile and Modica-Mortola]
    Let $\widetilde W(s)=\frac{1}{4}(s^2-1)^2$, and
    let $\tilde q:\mathbb{R}\rightarrow\mathbb{R}$ be the $1$-d optimal profile,
    that is,
    \[
      \begin{cases}
        \tilde q''=\widetilde W'(\tilde q), \\
        q(0)=0, \\
        \lim_{z\rightarrow\infty}q(z)=-1, \\
        \lim_{z\rightarrow-\infty}q(z)=1.
      \end{cases}
    \]
    Then $\tilde q(z)=\tanh(z)$.
    Let $q(z)=\frac{1}{2}\tanh(z)+\frac{1}{2}$.
    Then
    \[q(z)\leq Ce^{-|z|/C},\quad\text{as $z\rightarrow -\infty$.}\]
    Further let
    $s_{\Omega_0}(x)=\dist(x,\Omega_0)-\dist(x,\mathbb{R}^d\setminus\Omega_0)$
    be the signed distance function w.r.t.\ $\Omega_0$,
    and consider the one-parameter family of functions
    \[u_{\eps,0}^a\coloneqq q\left( \frac{-s_{\Omega_0}(ax)}{\eps} \right),\quad a\in(0,\infty).
    \]
    For $a=1$ we obtain the standard recovery sequence for $\chi_{\Omega_0}$
    in the $\Gamma$-convergence of $E_\eps$~\cite[Thm.\ I]{Modica1987}.
    In particular, they showed
    \[\limsup_{\eps\downarrow 0}E_\eps(u_{\eps,0}^1)\leq\sigma P(\Omega_0).\]
  \end{step}
  \begin{step}[Rescaling and volume constraint]
    We show that there exist $a_\eps\in(0,\infty)$ such that the functions $u_{\eps,0}^{a_\eps}$
    satisfy~\eqref{eq:initial-data-well-prepared},
    $u_{\eps,0}^{a_\eps}\in\mathcal{A}$ and
    $u_{\eps,0}^{a_{\eps}}\rightarrow\chi_{\Omega_0}$ in $L^1$.

    The one-parameter family $(u_{\eps,0}^a)_{a\in(0,\infty)}$ satisfies
    \begin{align*}
      \int_{\mathbb{R}^d}u_{\eps,0}^a(x)\,dx
      &=\int_{\mathbb{R}^d}u_{\eps,0}^1(ax)\,dx
        =\frac{1}{a}\int_{\mathbb{R}^d}u_{\eps,0}^1(x)\,dx.
    \end{align*}
    Hence there exists $a_\eps\in(0,\infty)$ such that $\|u_{\eps,0}^{a_\eps}\|_{L^1}=a_\eps^{-d}\|u_{\eps,
      0}^1\|_{L^1}=1$.

    Let $\tilde u_{\eps,0}\coloneqq u_{\eps,0}^{a_\eps}$.
    Since $u_{\eps,0}^1\rightarrow\chi_{\Omega_0}$ in $L^1$ and $\mathcal{L}^d(\Omega_0)=1$, we have
    \[a_\eps\xrightarrow{\eps\downarrow 0}1.\]
    Let $\chi_{\Omega_0,\eps}(x)\coloneqq\chi_{\Omega_0}(a_\eps x)$ and consider
    \begin{align*}
      \int_{\mathbb{R}^d}|\tilde u_{\eps,0}-\chi_{\Omega_0}|\,dx
      &\leq\int_{\mathbb{R}^d}|\tilde u_{\eps,0}-\chi_{\Omega_0,\eps}|\,dx
        +\int_{\mathbb{R}^d}|\chi_{\Omega_0,\eps}-\chi_{\Omega_0}|\,dx.
    \end{align*}
    For the first term, we have
    \begin{align*}
      \int_{\mathbb{R}^d}|\tilde u_{\eps,0}-\chi_{\Omega_0,\eps}|\,dx
      &=\int_{\mathbb{R}^d}|u_{\eps,0}^1(a_\eps x)-\chi_{\Omega_0}(a_\eps x)|\,dx \\
      &=\frac{1}{a_\eps^d}\int_{\mathbb{R}^d}|u_{\eps,0}^1(x)-\chi_{\Omega_0}(x)|\,dx.
    \end{align*}
    Here the RHS converges to zero as $\eps\downarrow 0$.
    For the second term, we have
    \[\chi_{\Omega_0,\eps}\xrightarrow{\eps\downarrow 0}
      \chi_{\Omega_0}\quad\text{in $L^1$,}
    \]
    because  $a_\eps\rightarrow 1$.
    Therefore $\tilde u_{\eps,0}\xrightarrow{\eps\downarrow 0}\chi_{\Omega_0}$
    in $L^1$.
  \end{step}
  \begin{step}[$\Gamma$-limsup inequality and moment bound]
    It remains to show that the second moment of $\tilde u_{\eps,0}$ is uniformly bounded in $\eps$ and that
    $\limsup_{\eps\downarrow 0}E_\eps(\tilde{u}_{\eps,0})\leq\sigma E_0(\chi_{\Omega_0})$.
    Since $a_\eps\rightarrow 1,\,q(z)\leq Ce^{-|z|/C}$ as $z\rightarrow -\infty$, and $\Omega_0$ is bounded,
    we can choose $R<\infty$ sufficiently large,
    and $\delta>0$ small enough such that $|a_\eps-1|<1/2$ for all $\eps<\delta$ and
    \[\tilde u_{\eps,0}(x)\leq Ce^{-|x|/2C\eps}
      \quad\text{for all $x\in\mathbb{R}^d\setminus B_R(0)$.}
    \]
    Then for any $p\geq 1$ and $\eps\leq\delta$
    \[M_p(\tilde u_{\eps,0})
      =\int_{\mathbb{R}^d}|x|^p \tilde u_{\eps,0}(x)\,dx
      \leq R^p\mathcal{L}^d(B_R(0))+C\int_{|x|>R}|x|^p e^{-|x|/2C\eps}\,dx.
    \]
    The RHS is uniformly bounded in $\eps$, hence
    $\sup_{\eps\leq\delta}M_2(\tilde{u}_{\eps,0})<\infty$.

    For the $\Gamma$-limsup inequality, observe that
    \begin{align*}
      E_{\eps}(\tilde u_{\eps,0})
      &=\int_{\mathbb{R}^d}\frac{\eps}{2}|\nabla\tilde u_{\eps,0}|^2\,dx
        +\int_{\mathbb{R}^d}\frac{1}{\eps}W(\tilde u_{\eps,0})\,dx \\
      &=\int_{\mathbb{R}^d}\frac{\eps}{2}|\nabla (u_{\eps,0}^1(a_\eps x))|^2\,dx
        +\int_{\mathbb{R}^d}\frac{1}{\eps}W(u_{\eps,0}^1(a_\eps x))\,dx \\
      &=\frac{1}{a_\eps^d}\int_{\mathbb{R}^d}\frac{\eps}{2}|\nabla u_{\eps,0}^1(x)|^2\,dx
        +\frac{1}{a_\eps^d}\int_{\mathbb{R}^d}\frac{1}{\eps}W(u_{\eps,0}^1(x))\,dx.
    \end{align*}
    Therefore
    \begin{align*}
      \limsup_{\eps\downarrow 0}E_{\eps}(\tilde u_{\eps,0})
      &\leq\limsup_{\eps\downarrow 0}(1+\delta)^{d}E_{\eps}(u_{\eps,0}^1) \\
      &\leq (1+\delta)^{d}\sigma P(\Omega_0;\mathbb{R}^d),
    \end{align*}
    for any $\delta>0$, because $a_\eps\rightarrow 1$ as $\eps\downarrow 0$ and hence
    $\limsup_{\eps\downarrow 0}E_\eps(u_{\eps,0}^1)\leq\sigma
    P(\Omega_0;\mathbb{R}^d)$.
    \qedhere
  \end{step}
\end{proof}

\section*{Acknowledgements}
	The present paper is an extension of the first author's master's thesis at the University of Bonn.
	This project has received funding from the Deutsche Forschungsgemeinschaft (DFG, German Research Foundation) under Germany's Excellence Strategy -- EXC-2047/1 -- 390685813.

\printbibliography{}

\end{document}